\documentclass[]{article}
\usepackage{graphicx}
\usepackage{graphics,color}
\usepackage{multicol}
\usepackage{multirow}
\usepackage{amssymb,amsmath,amsthm,epsfig}
\DeclareMathOperator*{\esssup}{ess\,sup}

\newcommand{\J}{P}

\newcommand{\wJ}{\widetilde P^{(\alpha,\alpha)}}

\newtheorem{theorem}{Theorem}

\newtheorem{corollary}{Corollary}

\newtheorem{remark}{Remark}
\usepackage{color}

\newcounter{reh}
\newcounter{rek}
\begin{document}

\begin{center}
{\large {\bf Multivariate nonparametric regression by least squares  Jacobi polynomials approximations}}\\
\vskip 0.5cm Asma BenSaber$^a$, Sophie Dabo-Niang$^b$ 
and  Abderrazek Karoui$^a$\footnote{
Emails: sophie.dabo@univ-lille.fr (S. Dabo-Niang), abderrazek.karoui@fsb.rnu.tn (A. Karoui), asmabensaber@gmail.com  (A. BenSaber)\\
This work was supported in part by the  
DGRST  research grant LR21ES10 and the PHC-Utique research project 20G1503.} 
\end{center}
\vskip 0.5cm
\noindent
$^a$  University of Carthage, Faculty of Sciences of Bizerte, Departmet of Mathematics, Jarzouna 7021,  Tunisia.\\
$^b$ University of Lille, UMR 9221, Lille F-59000  and  INRIA-MODAL Team, Haute-Borne, Villeneuve d’ascq, France.\\

\noindent{\bf Abstract}--- In  this work, we introduce and study a random  orthogonal projection based least squares estimator for the  stable solution
of a multivariate nonparametric   regression (MNPR) problem. More precisely, given an integer $d\geq 1$ corresponding to the dimension of the MNPR problem,
a positive integer $N\geq 1$ and a real parameter $\alpha\geq -\frac{1}{2},$ we show that  a fairly large class of  $d-$variate regression functions are well and stably approximated by its random projection over the orthonormal set of tensor product $d-$variate Jacobi polynomials with parameters $(\alpha,\alpha).$
 The associated  uni-variate  Jacobi polynomials have degree at most $N$ and their tensor products are  orthonormal over $\mathcal U=[0,1]^d,$ with respect 
 to the associated  multivariate Jacobi weights.
In particular, if we consider  $n$  random sampling points $\mathbf X_i$ following the $d-$variate Beta distribution, with parameters $(\alpha+1,\alpha+1),$
then  we give a relation involving $n, N, \alpha$ to ensure  that the  resulting  $(N+1)^d\times (N+1)^d$
random projection matrix is well conditioned. This is important in the sense that  unlike most least squares based estimators,  no extra regularization scheme
is needed by our proposed estimator.   Moreover, we provide squared integrated as well as  $L^2-$risk errors of this estimator. Precise estimates of these errors are given in the case where  the regression function belongs to an isotropic Sobolev space $H^s(I^d),$ with $s> \frac{d}{2}.$ Also, to handle the general and practical case of
an unknown  distribution of the $\mathbf X_i,$ we use Shepard's scattered interpolation scheme in order to generate fairly precise approximations
of the observed data at $n$ i.i.d. sampling points $\mathbf X_i$ following a $d-$variate Beta distribution. Finally, we illustrate the performance 
of our proposed multivariate nonparametric estimator  by some numerical simulations with synthetic as well as real data. \\


\section{Introduction}
 For an integer $d\geq 1$ and for sufficiently large positive integer $n,$ we consider the $d-$dimensional multivariate  nonparametric regression (MNR) model given by  
 \begin{equation}
 \label{model1}
 \mathbf Y_i = f(\mathbf X_i) + \varepsilon_i,\quad 1\leq i\leq n.
 \end{equation}
 Here,  the  $\mathbf Y_i$ are the $n$ random responses  and  $f: [0,1]^d \rightarrow \mathbb R$ is the real valued $d-$variate regression function. The ${\mathbf X_i=(X_{i,1},\ldots,X_{i,d})}$ are the $\ n$ random sampling vectors following a given joint probability distribution over $[0,1]^d.$ The $ \varepsilon_i$ are the  $n$ i.i.d. centered random variables with variance $\mathbb E\big[\varepsilon_i^2\big]=\sigma^2.$ 
In the sequel, we adopt the notation
\begin{equation}\label{notation}
\pmb k=(k_1,\ldots,k_d)\in \mathbb N_0^d,\quad \|\pmb k\|_\infty=\max_i |k_i|,\quad [[0,N]]^d=\{0,1,\ldots,N\}^d, \quad  \gamma_{\alpha,d}=\big(\beta(\alpha+1,\alpha+1)\big)^d. 
\end{equation}
Here,  $\alpha \geq -\frac{1}{2}$ is a real number and $\beta(\cdot,\cdot)$ is the usual beta function. For an integer $k\geq 0,$ we let $\wJ_k(x)$ denote the normalized Jacobi polynomial of degree $k$ and parameters $(\alpha,\alpha).$ The $\wJ_k(x), k\geq 0$  and satisfy the orthonormality relation  
\begin{equation}
\label{Eq0.1}
\int_I \wJ_j(x) \wJ_k(x) \, \omega_\alpha(x)\, dx = \delta_{j,k},\qquad I=[0,1],\quad \omega_\alpha(x)= x^\alpha (1-x)^\alpha.
\end{equation}
Under this notation, it is easy to check that the $d-$variate tensor product Jacobi polynomials family  
\begin{equation}\label{Eq1.1}
\Phi_{\pmb m}^\alpha(\pmb x)= \prod_{j=1}^d \wJ_{m_j}(x_j),\quad \pmb x=(x_1,\ldots,x_d),\quad \pmb m=(m_1,\ldots,m_d)\in \{0,1,\ldots\}^d=\mathbb N_0^d,
\end{equation}
forms an orthonormal basis of $L^2(I^d,\pmb \omega_\alpha),$ where ${\displaystyle \pmb \omega_\alpha(\pmb x) =\prod_{j=1}^d  \omega_\alpha(x_j)}.$ 
For a convenient positive integer $N,$ our proposed scheme is based on the approximation of the $d-$variate  regression function $f$ by its approximate projection  over the finite dimensional Hilbert subspace $\mathcal H_N$ of  $L^2(I^d,\pmb \omega_\alpha),$ given by
\begin{equation}\label{Eq1.2}
\mathcal H_N = \mbox{Span}\left\{ \Phi_{\pmb m}^\alpha(\pmb x),\,\, \pmb m\in [[0,N]]^d \right\}.
\end{equation}
We first  assume that  the $n$ i.i.d. random sampling vectors $\mathbf X_i$ follow the $d-$variate Beta distribution with density function
$h_{\alpha+1}(\cdot),$ given by 
\begin{equation}\label{Eq1.4}
h_{\alpha+1}(\pmb x)= \frac{1}{\Big(\beta(\alpha+1,\alpha+1)\Big)^d} \prod_{j=1}^d \omega_{\alpha}(x_j) \, \mathbf 1_{[0,1]^d}(\pmb x),
\end{equation}
where $\beta(\cdot,\cdot)$ is the usual Beta function. Nonetheless, we will see how our proposed multivariate estimator $\widehat  f^\alpha_N$ can be adapted in order to handle the more general and practical case where the $\pmb X_i$ follow an  unknown sampling distribution.
By assuming that 
$ \mathbf Y_i$ is approximated by the function $f^\alpha_N(\mathbf X_i),\, 1\leq i \leq n $, using an approximate of $f$ with the help of \eqref{Eq1.2}, the estimator  $\widehat f^\alpha_N(\cdot)$  of $ f$ is given by 
\begin{equation}\label{Eq1.3}
\widehat f^\alpha_N(\pmb x)=\sum_{\pmb m\in [[0,N]]^d} \widehat{C}_{\pmb m} \Phi_{\pmb m}^\alpha(\pmb x),\quad \pmb x\in [0,1]^d.
\end{equation}
Here,  the expansion coefficients vector
${\displaystyle \widehat{\pmb C}= \big(\widehat{C}_{\pmb m}\big)^T_{\pmb m \in [[0,N]]^d}}$  is computed in a stable way 
by the following formula,
\begin{equation}\label{solution1}
\widehat{\pmb C} = \big(G^\alpha_{d,N}\big)^{-1} \cdot \Big({\big(F^\alpha_{d,N}}\big)^T \cdot  \frac{\big(\beta(\alpha+1,\alpha+1)\big)^{d/2}}{n^{1/2}} \Big[ Y_{i} \Big]_{ i=1,...,n}^T.
\end{equation} 
The $n\times (N+1)^d$ random matrix  $F^\alpha_{d,N}$ and the $ (N+1)^d\times (N+1)^d$    positive definite random matrix $G^\alpha_{d,N}$, are given by 
\begin{equation}\label{matrixG}
G^\alpha_{d,N}= \big(F^\alpha_{d,N}\big)^T F^\alpha_{d,N},\qquad F^\alpha_{d,N}= \frac{\big(\beta(\alpha+1,\alpha+1)\big)^{d/2}}{  n^{1/2}} \Big[ \Phi_{\pmb m}^\alpha(\mathbf X_i)\Big]_{\substack{ \scriptstyle 1\leq i\leq n\\
 \scriptstyle  \pmb m\in [[0,N]]^d}}. 
\end{equation}
We should mention that in practice $N\ll n.$ Moreover, the scheme given by \eqref{Eq1.3}--\eqref{matrixG} is nothing but the least squares of the over-determined system ${\displaystyle \widehat f^\alpha_N(\pmb X_i)= \pmb Y_i,\,\, 1\leq i\leq n.}$
In other words, the estimator $\widehat f^\alpha_N$ is a solution of the minimization problem
\begin{equation}\label{minimization}
\widehat f^\alpha_N = \arg\min_{f \in \mathcal H_N} \sum_{i=1}^n \big(f(X_i)-Y_i\big)^2.
\end{equation}

Note that our proposed estimator $\widehat f^\alpha_N$ belongs to a larger class of  multivariate orthogonal polynomials based least-squares estimators for multivariate nonparametric regression problems, see for example \cite{Cohen, Guo, Shin1, Zhou2}. Moreover, $\widehat f^\alpha_N$  is closely related to the  generalized Polynomial Chaos (gPC) or the Polynomial Chaos Expansions (PCE) class of nonparametric regression estimators in higher dimensions, see for example \cite{Blatman, Guo, Jakeman, Shin1, Torre, Xiu, Zhou2}. The gPC or the PCE techniques aim to approximate a $d-$variate function $f$ via $d-$orthogonal polynomials, where the orthogonality is defined by a probability measure on the input space $\mathcal X\subset \mathbb R^d.$ This technique is widely used in the area of parametric uncertainty quantification, where one is faced with the  challenge to approximate functions in high dimension $d$ and via its point evaluations. In the literature, there exist various techniques for gPC based multivariate regression. Among these techniques, we cite least-squares, weighted discrete  least-squares, sparse grids approximations, $l_1-$minimization sparse approximation. For more details, the reader is refereed to \cite{Zhou2} and the references therein. In general, three types of finite dimensional  multivariate polynomials spaces are used by gPC type regression schemes. More precisely, for a positive integer $N,$ these spaces are:\\
$-$ $\mathcal P_N^{TP}=\mbox{Span}\{\pmb x^{\pmb i}=x_1^{i_1}\cdots x_d^{i_d},\quad \|\pmb i\|_\infty =\max_j |i_j|\leq N \}:$ the tensor product space of degree $N.$\\
$-$ $\mathcal P_N^{TD}=\mbox{Span}\{\pmb x^{\pmb i}=x_1^{i_1}\cdots x_d^{i_d},\quad \|\pmb i\|_1=\sum_{j=1}^d |i_j| \leq N \}:$ the total degree  space of degree $N.$\\
$-$ For $0<q<1,$  $\mathcal P_{q,N}^{HC}=\mbox{Span}\{\pmb x^{\pmb i}=x_1^{i_1}\cdots x_d^{i_d},\quad \|\pmb i\|_q=\Big(\sum_{j=1}^d |i_j|^q \Big)^{1/q}\leq N \}:$ the hyperbolic cross space of  degree $N.$\\
Note that the dimensions of the first two spaces are given by $\dim(\mathcal P_N^{TP})=(N+1)^d,\,\,  \dim(\mathcal P_N^{TD})={N+d \choose d}.$ For the space 
 $\mathcal P_{q,N}^{HC},$ there is no precise estimate of its dimension, nonetheless numerical evidences indicate that for the values of $0<q \leq 0.5,$ 
and even for moderate large values of the dimension $d,$ $\dim(\mathcal P_{q,N}^{HC}) $ is drastically much smaller than the dimensions of the previous two spaces. For more details, the reader is refereed to \cite{Blatman}. In the present work, we restrict ourselves to the case of the tensor product polynomial space, generated by the product of univariate Jacobis polynomials. Nonetheless, most of our results can be extended to the other two multivariate polynomial spaces $\mathcal P_N^{TD}$ and $\mathcal P_{q,N}^{HC}.$ Note that  due to the blow-up of its cardinality with respect to the dimension $d,$  the tensor product space $\mathcal P_N^{TP}$ is practical only for small values of the dimension $d.$
For moderate large values of $d,$ one has to apply dimension reductions techniques, such as  sparsity and / or more optimal sampling techniques, see for example \cite{Blatman, Guo,  Jakeman, Shin1, Zhou2}. Also, a popular class of nonparametric regression  estimators adapted for moderate values of the dimension $d$ are based on the combination of a model selection and smoothing (regularization) tecnhiques, see for example \cite{Lin} and the references therein. \\
Perhaps, the stability problem is one of the most important issues related to the multivariate least-squares regression schemes. In  the literature, only very few references have dealt with this issue so far, see for example \cite{Cohen, Zhou2}. 
In particular, in \cite{Zhou2} the authors have studied the stability of weighted least-squares with random sub-sampling of tensor Gauss points. Under our notation,
they have shown that for any $\mu \geq 2,$ the previous scheme is stable with probability at least ${\displaystyle 1-\frac{2}{n^{\mu-1}},}$ provided that 
\begin{equation}
\label{stabilitycond1}
\frac{n}{\log n} \gtrsim \frac{4\mu}{c} C_w C_b \dim (\mathcal P_{\Lambda_N}).
\end{equation}
 Here, $\mathcal P_{\Lambda_N}$ is one of the three polynomial spaces, $\Lambda_N$ is the associated set of indices, $c$ is a uniform positive constant, $C_w$ is a constant depending on the Gauss weights and ${\displaystyle C_b= \max_{\pmb i\in \Lambda_N} \pmb \psi_{\pmb i}(\pmb z)},$ where the $\pmb \psi_{\pmb i}$ are the different multivariate polynomials and the $\pmb z$ are the different tensor product Gauss nodes. The quantities $C_b$ and $\dim (\mathcal P_{\Lambda_N})$ have the largest contributions to the previous lower bound of the stability condition. Precise estimates of these two  quantities and consequently of the stability condition have been given in \cite{Zhou2} for the special cases of the tensor product Legendre, as well as the Chebyshev polynomials.\\
One of the main results of this work is to prove that under a condition relating the parameters $N, d$ and $\alpha\geq -\frac{1}{2},$ our  least-squares polynomials regression,  with random sampling following a multivariate beta distribution is stable.  More precisely, 
if ${\displaystyle \kappa_2(G^\alpha_{d,N})=\frac{\lambda_{\max}(G^\alpha_{d,N})}{\lambda_{\min}(G^\alpha_{d,N})}}$
is the $2-$norm condition number of the random matrix $G^\alpha_{d,N},$ then for $  \alpha \geq -\frac{1}{2}$ and for any $0< \delta <1,$
we have 
\begin{equation}\label{kappa2G}
\mathbb P \left( \kappa_2 (G^\alpha_{d,N}) \leq \frac{1+\delta}{1-\delta} \right) \geq 1 - 2(N+1)^d \exp\left(-\frac{n\, \delta^2}{3 \big(B_\alpha (N+1)^{2\alpha +2}\big)^d}\right).
\end{equation}
Here, $B_{\alpha} \leq 1$ is a constant depending only on $\alpha.$ From the previous estimate, it can be easily checked that our estimator is stable with high probability whenever  $n,$ the total  number of sampling points  satisfies
\begin{equation}
\label{Eq1.5}
 n\geq \frac{3}{\delta^2}  \log\big(2(N+1)^d\big)\Big( B_\alpha (N+1)^{2\alpha+2} \Big)^{d}.
\end{equation}
Note that for the case of the multivariate  tensor product space based on Jacobi  polynomials, the previous stability is a refinement of the more general purpose 
stability condition \eqref{stabilitycond1}. Moreover, from the inequality \eqref{Eq1.5}, the special value of ${\displaystyle \alpha=-\frac{1}{2}}$ is convenient in the sense that it ensures 
the stability of the estimator $\widehat f^\alpha_N$ with the smaller  values of $n.$

 A second main result of this work is the following weighted $L^2(I^d)-$error of the estimator $\widehat  f^\alpha_N.$ For this purpose, we use the fairly usual assumption 
 on the i.i.d. random noises $(\varepsilon_i)_i,$  that for a given probability value
 $0< p_\epsilon \ll 1,$ there exists a moderate positive constant $M_\epsilon $ so that 
 \begin{equation}
 \label{Eqq0.2}
 \mathbb P \Big( |\varepsilon_i| > M_\epsilon \Big) = p_\epsilon,\quad 0< p_\epsilon \ll 1.
 \end{equation}
Let  $\pi_N f$ be the orthogonal projection of $f$ over $\mathcal H_N$  and let
$\|\cdot \|_\alpha$ be the usual $2-$norm of $L^2(I^d,\pmb \omega_{\alpha}).$  Then, for any ${\displaystyle 0< \delta < \frac{1}{\|\pi_N(f)\|_\alpha}}$, we have with high probability depending on $\delta,$
\begin{equation}
\label{Eq0.2}
 \|f-\widehat f^\alpha_N\|_\alpha\leq \|f-\pi_N(f)\|_\alpha+\sqrt{\kappa_2(G^\alpha_{d,N})}
 \Big(\|f-\pi_N(f)\|^2_\infty+\sigma^2+\delta\Big)^{1/2}\frac{1}{\sqrt{1-\frac{\delta}{\|\pi_N(f)\|_\alpha}}},
\end{equation}
Precise  estimates for the errors
$\|f-\pi_N(f)\|_\alpha$ and $\|f-\pi_N(f)\|_\infty$ will be given for those 
regression function  $f$ belonging to an isotropic Sobolev space $H^s(I^d),$ with 
$s> d(\alpha+\frac{3}{2}).$  More importantly, for $0<\delta<1,$  we give the following $L^2-$risk error of $\widehat F_N$ (see equation (\ref{1'})), a truncated version of the  estimator  $\widehat f^\alpha_N,$
\begin{eqnarray}
\label{Eq0.3}
\mathbb E\Big[\| f-\widehat F_N \|_\alpha^2\Big]&\leq&\frac{(N+1)^d}{n(1-\delta)^2}\Bigg(\sigma^2+\big(\eta_\alpha^2(N+1)^{2\alpha+1}\big)^d\| f-\pi_Nf \|_\alpha^2\Bigg)+\| f-\pi_Nf \|_\alpha^2\nonumber\\
&&\qquad +
 4M_f^2(N+1)^d\big(\beta(\alpha+1,\alpha+1)\big)^{d} \exp\Big(-\frac{n \delta^2}{2 \big(B_\alpha (N+1)^{2\alpha +2}\big)^d}\Big),
\end{eqnarray}
where $\eta_\alpha, B_\alpha$ are constants depending only on $\alpha$ and $|f(\pmb x)|\leq M_f,\, a.e.\,  \pmb x\in I^d.$\\

It is interesting to note that the curse of dimension does not 
only affect the computational load that grows drastically with the dimension $d,$ but also has a negative effect on the convergence rate for the multivariate regression estimators. In \cite{Gyorfi}, the authors have given a detailed study of this issue. In particular, for the case of a multivariate nonparametric problem with the $n$ random sampling vectors $\pmb X_i$ belonging to a compact subset $\pmb \chi$  of 
$\mathbb R^d,$ and for $L>0$ and a positive integer $\alpha\geq 1,$ they have considered 
the H\"older class of $d-$variate functions, defined by 
$$ H_d(\alpha,L)=\Big\{ g: |D^{\pmb s} g(\pmb x)-D^{\pmb s} g(\pmb y)|\leq L \|\pmb x-\pmb y\|,\quad \pmb x,\pmb y\in \pmb \chi,\quad \|\pmb s\|_1\leq \alpha-1 \Big\}.$$
Here, $D^{\pmb s}g$ denotes the partial derivatives of $g$ of order $\|\pmb s\|_1$ and associated with $\pmb  s \in \mathbb N_0^d.$
Then, it has been shown in \cite{Gyorfi}, see also \cite{Ryan}, that for most of the usual nonparametric regression estimator $\widehat g,$ the min-max convergence rate over the functional space $H_d(\alpha,L)$ is given by 
$$\inf_{\widehat  g} \sup_{g_0\in H_d(\alpha,L)} \mathbb E\Big(\|\widehat g-g_0\|^2\Big) \gtrsim  n^{-2\alpha/(2\alpha+d)}.$$
That is the optimal convergence rate $O\Big(n^{-2\alpha/(2\alpha+d)}\Big)$   decays in a significant manner as the dimension $d$ grows.\\
In this work, we prove that our proposed estimator  has a similar optimal convergence rate under the condition that the regression function belongs to an isotropic Sobolev space $H^s(I^d),$ for some $s>0.$ More precisely, we prove that in this case, the $L_2-$risk of the proposed estimator is of order  $O\Big(n^{-2s/(2s+d)}\Big).$\\
  
This work is organized as follows. In section 2, we give some mathematical preliminaries that will be frequently used to prove the different results of this work. In section 3, we first prove the  stability property  of our proposed nonparametric regression estimator. Then,  we study the convergence rate of our estimator. Section 4 is devoted to various numerical simulations ( on  synthetic data as well as real data) that illustrate the different  results of this work.
Finally, in section 5, we give some concluding remarks concerning this work.

\section{Mathematical preliminaries}\label{maths.prelim} 

In this paragraph, we provide the reader with some mathematical preliminaries that are frequently used to describe and prove the different results of this work. we first give the following fairly known definitions and properties related to the uni-variate Jacobi polynomials.\\

It is well known, see for example \cite{Andrews} that for any two real parameters  $\alpha, \beta >-1,$ the classical Jacobi polynomials $\J_k$ are defined for $x\in [-1,1],$ by the following Rodrigues formula,
\begin{equation}
\label{Jacobi_Rodrigues}
\J_k^{(\alpha,\beta)} (x)= \frac{(-1)^k}{2^k k!} \frac{1}{\omega_{\alpha,\beta}(x)} \frac{\mbox{d}^k}{\mbox{d} x^k} \Big( \omega_{\alpha,\beta}(x)(1-x^2)^k\Big),\quad
\omega_{\alpha,\beta}(x)= (1-x)^\alpha (1+x)^\beta,
\end{equation}
with
$$\J_k^{(\alpha,\beta)}(1) = { k+\max(\alpha,\beta) \choose{k}}  =\frac{\Gamma(k+\max(\alpha,\beta)+1)}{k!\, \Gamma(\max(\alpha,\beta)+1)}.$$
Here, $\Gamma(\cdot)$ denotes the usual Gamma function. 
More importantly and as others families of classical orthogonal polynomials,
the Jacobi polynomials $\J_k$ are given by the following practical three
term recursion formula
\begin{equation}\label{recursion1}
\J_{k+1}^{(\alpha,\beta)}(x)= (a_k x + b_k)\J_{k}^{(\alpha,\beta)}(x) -c_k \J_{k-1}^{(\alpha,\beta)}(x),\quad x\in [-1,1].
\end{equation}
with  $\J_0^{(\alpha,\beta)}(x)=1,\quad \J_1^{(\alpha,\beta)}(x)=\frac{1}{2}(\alpha+\beta+2)x +\frac{1}{2}(\alpha+\beta).$  Here, 
\begin{eqnarray}\label{recursion2}
a_k &=&\frac{(2k+\alpha+\beta+1)(2k+\alpha+\beta+2)}{2 (k+1)(k+\alpha+\beta+1)},\quad b_k=\frac{(\alpha^2-\beta^2)(2k+\alpha+\beta+1)}{2(k+1)(k+\alpha+\beta+1)(2k+\alpha+\beta)}\nonumber\\
c_k&=&\frac{(k+\alpha)(k+\beta)(2k+\alpha+\beta+2)}{(k+1)(k+\alpha+\beta+1)(2k+\alpha+\beta)}.
\end{eqnarray}
These Jacobi polynomials satisfy the following orthogonal relation 
$$\int_{-1}^1 \J_{k}^{(\alpha,\beta)}(x) \J_{m}^{(\alpha,\beta)}(x) \omega_{\alpha,\beta}(y)\, dy =h_k^{\alpha,\beta} \delta_{k,m},\quad 
h_k^{\alpha,\beta}=\frac{2^{\alpha+\beta+1}\Gamma(k+\alpha+1)\Gamma(k+\beta+1)}{k!(2k+\alpha+\beta+1)\Gamma(k+\alpha+\beta+1)},$$
where $\delta_{k,m}$ is the usual Kronecker delta function. Note that the set $\{ \frac{1}{\sqrt{h_k^{\alpha,\beta}}}\J_k^{(\alpha,\beta)}(x),\,k\geq 0\}$ is an orthonormal basis of the weighted $L^2([-1,1],\omega_{\alpha,\beta}),$ which is a Hilbert space associated with the inner product $<\cdot,\cdot>_\omega,$ defined by ${\displaystyle 
<f,g>_\omega=\int_{-1}^1 f(t) g(t) \omega_{\alpha,\beta}(t)\, dt}.$

In the sequel, we will only consider the case of $\beta=\alpha\geq -\frac{1}{2}.$ We let $\wJ_k$ denote the normalized Jacobi polynomial over $I=[0,1]$ of degree $k$ and given by 
\begin{equation}\label{JacobiP}
\wJ_{k}(x)= \frac{2^{\alpha+1/2}}{\sqrt{h_k^{\alpha,\beta}}}\J_k^{(\alpha,\alpha)}(2x-1),\quad x\in I.
\end{equation}
In this case, we have
\begin{equation}\label{Normalisation1}
\| \wJ_k \|^2_\omega=\int_{0}^1 (\wJ_k(y))^2 \omega_{\alpha}(y)\, dy =1,\quad \omega_{\alpha}(y)=y^\alpha (1-y)^\alpha.
\end{equation}
The following   useful upper bounds for the normalized Jacobi polynomials $\wJ_k,$ for $k\geq 2$
is borrowed from \cite{BenSaber-Karoui}

\begin{equation}\label{Ineq2.1}
\max_{x\in [-1,1]} |\wJ_k (x)| \leq \eta_{\alpha}  k^{\alpha} \sqrt{k+\alpha+\frac{1}{2}},\quad  \forall\, k\geq 2,
\end{equation}
where
\begin{equation}\label{IneQ2.1}
 \eta_{\alpha}= \frac{\sqrt{2}}{\Gamma(\alpha+1)} \exp\left(\frac{\max(0,\alpha)}{6}+\frac{\alpha^2}{4}\right).
\end{equation}

Next, we briefly describe the original as well as modified Shepard's algorithms for the interpolation of multivariate scattered data. This type of interpolation is needed in order to get convenient interpolations of the observations at some appropriate random sampling points. The original Shepard's interpolation of multivariate scattered data can be described as follows, see for example  \cite{Shepard}. Let 
$\mu>0$ be a  positive  real number  and let  $\pmb X=\{\pmb x_1,\ldots,\pmb x_n\}$
be a  set of $n$ distinct points of a domain $\mathcal D\subset \mathbb R^d,\, d\geq 1,$ then  for the $n$ associated  real valued function evaluations $f_i=f(\pmb x_i)$, the Shepard's interpolation operator is given by 
\begin{equation}\label{Ineq2.2}
S_\mu(f)(x) =\sum_{i=1}^n A_{\mu,i}(\pmb x) f_i,\qquad A_{\mu,i}(\pmb x) =\frac{\Big( \mbox{d}(\pmb x,\pmb x_i)\Big)^{-\mu}}{\sum_{i=1}^n \Big( \mbox{d}(\pmb x,\pmb x_i)\Big)^{-\mu} },
\end{equation}
where $\mbox{d}(\cdot,\cdot)$ is a distance on $\mathbb R^d.$
The previous interpolation has the drawback that the basis functions $A_{\mu,i}$ have significant values at those points $\pmb x_i$ which are far from the considered interpolation point $\pmb x.$ To overcome this drawback, see for example \cite{Shepard},  the $A_{\mu,i}$ are substituted by compactly supported basis functions $W_{\mu,i},$
so that for a given radius of influence $R>0,$ the modified Shepard's algorithm is given by 
\begin{equation}\label{Ineq2.3}
S_\mu^R(f)(x) = \sum_{i=1}^n W_{\mu,i}(\pmb x) f_i,\qquad W_{\mu,i}(\pmb x) =\frac{\Big( \frac{1}{\mbox{d}(\pmb x,\pmb x_i)}-\frac{1}{R}\Big)_+^{\mu}}{\sum_{i=1}^n \Big( \frac{1}{\mbox{d}(\pmb x,\pmb x_i)}-\frac{1}{R}\Big)_+^{\mu}},
\end{equation}
where, $(t)_+=\max(t,0).$ Also,  the original Shepard's algorithm has been further developed by considering  a
combined Shepard-Multivariate Taylor interpolation polynomial $S_{\mu,r}$, where for an integer $r\geq 0,$ 
\begin{equation}\label{Ineq2.4}
S_{\mu,r}(f)(x)=\sum_{i=1}^n A_{\mu,i}(\pmb x) T_{r,\pmb x_i}(f)(\pmb x),\qquad  T_{r,\pmb x_i}(f)(\pmb x)=\sum_{\nu=0}^r \frac{D^\nu f(\pmb x_i)}{\nu !} (\pmb x -\pmb x_i)^\nu.
\end{equation}
It has been showed that 
\begin{equation}\label{Ineq2.5}
\|S_{\mu,r}(f)-f\| = \left\{\begin{array}{ll} O(h^{r+1}) &\mbox{ if } \mu-d >r+1\\ O(h^{\mu-d} |\log h|) &\mbox{ if } \mu-d=r+1.\end{array} \right.
\end{equation} 
Here, $h$ is the mesh step size that is the largest distance between the neighboring points $\pmb x_i.$ For more details, the reader is refereed to \cite{Shepard} or to \cite{Lodha}.
In particular, the last reference is a comprehensive review of several  other  types of scattered  multivariate interpolation techniques.\\

Next, we give the following highly  useful matrix Chernoff theorem, see for example [\cite{Tropp},\, p.10], that provides us with  upper and lower bounds for the smallest and largest eigenvalues of a sum of random Hermitian matrices. \\

\noindent
{\bf Matrix Chernoff Theorem:}  Consider a finite sequence of $n$ independent $D\times D$ random Hermitian matrices $\{ \pmb Z_k\}.$ Assume that for some $B>0,$ we have 
$$0 \preccurlyeq \pmb Z_k \preccurlyeq B I_D.  $$ Let 
$$\mathbf A= \sum_{k=1}^n \pmb Z_k,\quad \mu_{\min}= \lambda_{\min} \big(\mathbb E(\mathbf A) \big),\quad  \mu_{\max}= \lambda_{\max} \big(\mathbb E (\mathbf A) \big).$$
Then,  for any $\delta \in (0,1],$ we have
\begin{equation}\label{Chernoff}
\mathbb P \Big(\lambda_{\min} \big(\mathbf A\big)\leq (1-\delta)\mu_{\min}\Big)\leq D \left[\frac{e^{-\delta}}{(1-\delta)^{1-\delta}}\right]^{\mu_{\min}/B},\quad \mathbb P \Big(\lambda_{\max} \big(\mathbf A\big)\geq (1+\delta)\mu_{\max}\Big)\leq D \left[\frac{e^{\delta}}{(1+\delta)^{1+\delta}}\right]^{\mu_{\max}/B}.
\end{equation}
As a consequence of the previous two inequalities, one can check, see for example [\cite{Tropp},\, p.12], that for $\delta\in (0,1],$ we have
\begin{equation}\label{Eqq3.7}
\mathbb P\Big(\lambda_{\min}(\mathbf A)\leq (1-\delta) \mu_{\min} \Big) \leq D \exp\Big(-\frac{\delta^2 \mu_{\min}}{2 B}\Big),\quad 
\mathbb P\Big(\lambda_{\max} (\mathbf A)\geq (1+\delta)\mu_{\max} \Big) \leq D \exp\Big(-\frac{\delta^2 \mu_{\max}}{3 B}\Big).
\end{equation}

Finally, the following Gershgorin circle theorem, see for example \cite{Johnson},  will be needed to prove one of the main theoretical results of this work which is an estimate for an upper bound of the random projection matrix.\\

\noindent
{\bf Gershgorin circle Theorem:}  Let $A= [a_{ij}]$ be a complex $n\times n$  matrix. For $1\leq i\leq n,$ let 
${\displaystyle R_i = \sum_{j\neq i} |a_{ij}|.}$ Then, every eigenvalue of $A$ lies within at least one of the discs $D(a_{ii}, R_i).$\\

\section{Stability and convergence rates of the estimator}

In this paragraph, we first describe our orthogonal projection based scheme for solving the MNPR problem \eqref{model1}. Then, we prove the inequality
\eqref{kappa2G}. That is with high probability,  the random projection matrix $G^\alpha_{d,N},$ given by \eqref{matrixG} is well conditioned. Our random orthogonal projection based scheme is described as follows. We first substitute in \eqref{Eq1.3}, $\pmb x$ by $\mathbf X_i,\, i=1,...,n,$ where the $\mathbf X_i$ are i.i.d random samples following the $d-$dimensional Beta probability distribution $h_{\alpha+1}(\cdot),$ given by \eqref{Eq1.4}. Note that the beta distribution
is widely used in the framework of various models from mathematical statistics, see for example \cite{Guolo, Jakeman, Zhou2} . Then, after rescaling by the factor $\big( n \beta(\alpha+1,\alpha+1)\big)^{-d/2},$ one obtains the following overdetermined system of $n$ equations in the 
$(N+1)^d$ unknown expansion coefficients vectors $\widehat C_{\pmb m},$
\begin{equation}\label{system1}
\frac{\big(\beta(\alpha+1,\alpha+1)\big)^{d/2}}{n^{1/2} }Y_i= F^\alpha_{d,N} \cdot \widehat{\pmb C},\qquad \widehat{\pmb C}= \big(\widehat{C}_{\pmb m}\big)^T_{\pmb m \in [[0,N]]^d},
\end{equation}
where the random matrix $F^\alpha_{d,N}$ is given by \eqref{matrixG}. Note that since the basis functions $\Phi^\alpha_{\pmb m}(\cdot),\, \pmb m\in [[0,N]]^d$ form an orthonormal basis of the finite dimensional subspace $\mathcal H_N$ of $E=L^2(I^d,\pmb \omega_\alpha(\pmb x) d\pmb x),$ then for 
$f(\cdot)\in E,$ its orthogonal projection over $\mathcal H_N$ denoted by $\pi_N f$ is uniquely defined. The expansion coefficients of $\pi_N f$ are given by 
\begin{equation}\label{expansion_coeffs}
C_{\pmb m}= < f, \Phi^\alpha_{\pmb m}> =\int_{I^d} f(\pmb x)\Phi^\alpha_{\pmb m}(\pmb x) \pmb \omega_\alpha(\pmb x) d\pmb x,\qquad \pmb m\in [[0,N]]^d.
\end{equation}
That is under the hypothesis of a noise free regression model with a piecewise continuous regression function, it can be easily checked that for sufficiently large value of the total sampling points  $n,$ the least square norm solution of \eqref{system1} doest not depend on the considered sampling 
set $\{\mathbf X_i,\, i\in [[1,n_1^d]]=[[1,n]]\}.$ Multiplying the system \eqref{system1} from both sides by $\big(F^\alpha_{d,N}\big)^T.$ As we show in the sequel, with high probability, the $(N+1)^d-$dimensional random matrix $G^\alpha_{d,N}= \big(F^\alpha_{d,N}\big)^T\cdot F^\alpha_{d,N}$ is positive definite 
and hence invertible. Consequently, by applying $\big(G^\alpha_{d,N}\big)^{-1}$ to the previous intermediate system \eqref{system1}, one gets the reduced Cramer system
\eqref{solution1}.\\

The following theorem is one of the main results of this work. It gives us with high probability, a relatively small  upper bound for $\kappa_2(G^\alpha_{d,N}),$ the  $2-$norm condition number  of the positive definite random projection matrix, associated to the Jacobi system. 
This theorem allows us to use the inverse of $G^\alpha_{d,N}$ in computing the estimator $\widehat f^\alpha_N(\cdot),$ given by \eqref{Eq1.3} and \eqref{solution1}.

\begin{theorem}\label{Mainthm1}
Under the previous notation and assumption, for any $\alpha\geq -\frac{1}{2}$ and any $0< \delta <1,$
we have  
\begin{equation}\label{kkappa2G}
\mathbb P \left( \kappa_2 (G^\alpha_{d,N}) \leq \frac{1+\delta}{1-\delta} \right) \geq 1 - 2(N+1)^d \exp\left(-\delta^2\frac{n}{3 \big(B_\alpha (N+1)^{2\alpha +2}\big)^d}\right).
\end{equation}
Here, $B_{\alpha} \leq 1$ is a constant depending only on $\alpha.$
\end{theorem}

\noindent{\bf Proof:} To alleviate notation, we consider the correspondence 
$g: [[0,N]]^d \rightarrow [[1, (N+1)^d]],$ given by 
$$g(m_1,\ldots,m_d)= 1+ \sum_{k=1}^d m_k (N+1)^{k-1},\qquad \pmb m=(m_1,\ldots,m_d)\in [[0,N]]^d.$$
Also, we shall use the notation
$$\Psi^\alpha_m(\pmb x)=  \Phi^\alpha_{g^{-1}(m)}(\pmb x),\qquad m\in [[1,(N+1)^d]].$$
So that we have 
\begin{equation}
\label{matrixG2}
G^\alpha_{d,N} = \left[ \frac{\big(\beta(\alpha+1,\alpha+1)\big)^d}{n } \sum_{i=1}^n \Psi^\alpha_{m_1}(\mathbf X_{i})\Psi^\alpha_{m_2}(\mathbf X_{i})\right]_{\substack{ \scriptstyle 1\leq m_1\leq (N+1)^d\\
 \scriptstyle 1\leq m_2\leq (N+1)^d}}.
\end{equation}
It is easy to see that for $i=1,\ldots,n$ and $1\leq j\leq d,$ we have 
$$ \mathbb E\Big( \wJ_k(X_{i,j}) \wJ_l(X_{i,j}) \Big)=\int_{I} \wJ_k(x) \wJ_l(x) \frac{\omega_\alpha(x)}{\beta(\alpha+1,\alpha+1)}\, dx = \frac{1}{\beta(\alpha+1,\alpha+1)} \delta_{k,l}.$$
Consequently, we have 
$$ \mathbb E\Big(G^\alpha_{d,N}\Big)= I_{(N+1)^d},$$
the $(N+1)^d-$dimensional identity matrix. On the other hand, we write $G^\alpha_{d,N}$ as follows 
$$G^\alpha_{d,N} = \sum_{i=1}^n H_i^\alpha,\qquad H_i^\alpha= \left[ \frac{\big(\beta(\alpha+1,\alpha+1)\big)^d}{n }  \Psi^\alpha_{m_1}(\mathbf X_{i})\Psi^\alpha_{m_2}(\mathbf X_{i})\right]_{\substack{ \scriptstyle 1\leq m_1\leq (N+1)^d\\
 \scriptstyle 1\leq m_2\leq (N+1)^d}}.$$
We check that there exists a constant $B$ such that 
\begin{equation}
\label{Eq3.2}
0 \preccurlyeq H_i^\alpha \preccurlyeq B  I_{(N+1)^d},\qquad i=1,\ldots,n,
\end{equation}
that is $H_i^\alpha$ and $ B I_{(N+1)^d}-H_i^\alpha$ are positive semi-definite.  Since $H_i^\alpha = A_i^T A_i,$ where $A_i$ is the $(N+1)^d\times 1$ matrix given by ${\displaystyle A_i= \frac{(\beta(\alpha+1,\alpha+1)^{d/2})}{n^{1/2}}}\Big[\Psi_1^\alpha(\mathbf X_i)\cdots\Psi_{(N+1)^d}^\alpha(\mathbf X_i)\Big],$ then its different eigenvalues $\lambda_k(H_i^\alpha) $ are non-negative, that is $H_i^\alpha \succcurlyeq 0.$ To prove the second inequality of \eqref{Eq3.2}, we note that from Gershgorin circle theorem, we have 
\begin{equation}\label{Eq3.3}
\lambda_{\max}(H_i^\alpha)\leq \frac{(\beta(\alpha+1,\alpha+1))^{d}}{n} \max_{1\leq m_1\leq (N+1)^d}\Big(|\Psi^\alpha_{m_1}(\mathbf X_i)|^2+
\sum_{\substack{ \scriptstyle m\neq m_1\\ 1\leq m\leq (N+1)^d}}  |\Psi^\alpha_{m_1}(\mathbf X_i)\Psi^\alpha_{m}(\mathbf X_i)|\Big).
\end{equation}
On the other hand, from the upper bound  for the normalized Jacobi polynomials $\wJ_k(\cdot),$ given by \eqref{Ineq2.1}, we have 
\begin{eqnarray*}\label{Eq3.4}
\max_{\substack{ \scriptstyle 1\leq m\leq (N+1)^d\\ 1\leq i\leq n}} |\Psi_m^{\alpha}(\mathbf X_i)| &\leq & \| \Psi_{(N+1)^d}^{\alpha}\|_{\infty} =\prod_{k=1}^d \|\wJ_{N+1}\|_{\infty}\nonumber \\
&\leq& \eta_\alpha^d \Big( N^\alpha \sqrt{N+\alpha+\frac{1}{2}}\Big)^d \leq \Big(\eta_\alpha (N+1)^{\alpha+\frac{1}{2}}\Big)^d
\end{eqnarray*}
That is for $i=1,\ldots,n,$ we have 
\begin{eqnarray}\label{Eq3.5}
\max_m  |\Psi_{m_1}^{\alpha}(\mathbf X_i)|^2+\sum_{m\neq m_1}  |\Psi_{m}^{\alpha}(\mathbf X_i) \Psi_{m_1}^{\alpha}(\mathbf X_i)| &\leq& (N+1)^d \Big(\eta_\alpha^2 (N+1)^{2\alpha+1}\Big)^d= \Big(\eta_\alpha^2(N+1)^{2\alpha+2}\Big)^d 
\end{eqnarray} 
By combining \eqref{Eq3.3} and \eqref{Eq3.5}, one gets the second inequality of \eqref{Eq3.2} with 
\begin{equation}\label{Eq3.6}
B=\frac{(\beta(\alpha+1,\alpha+1))^{d}}{n}\Big(\eta_\alpha^2(N+1)^{2\alpha+2}\Big)^d.
\end{equation}
Next, we apply the following estimate from \cite{Tropp}, for the minimum and the maximum eigenvalue of a sum of positive semi definite random matrices.
If ${\displaystyle \mathbf A =\sum_{k=1}^n \mathbf Z_k,}$ where the $\mathbf Z_k$ are $D\times D$ random matrices satisfying $0 \preccurlyeq \mathbf Z_k \preccurlyeq  B I_n,$ for  some positive constant $B$ and if ${\displaystyle \mu_{\min}=\lambda_{\min} \Big(\mathbb E(\mathbf A)\Big),\,\,\, \mu_{\max}=\lambda_{\max} \Big(\mathbb E(\mathbf A)\Big),}$ then from \eqref{Eqq3.7}, we have 
\begin{equation}\label{Eq3.7}
\mathbb P\Big(\lambda_{\min}(\mathbf A)\geq (1-\delta) \mu_{\min} \Big) \geq 1-D \exp\Big(-\frac{\delta^2 \mu_{\min}}{2 B}\Big),\quad \forall\,\, \delta\in (0,1] 
\end{equation}
and
\begin{equation}\label{Eq3.8}
\mathbb P\Big(\lambda_{\max} (\mathbf A)\leq (1+\delta)\mu_{\max} \Big) \geq 1-D \exp\Big(-\frac{\delta^2 \mu_{\max}}{3 B}\Big),\quad \forall\,\, \delta >0.
\end{equation}
In the special case where ${\displaystyle \mathbf Z_k = H_k^\alpha,}$ are the $n$ previous  $(N+1)^d\times (N+1)^d$ positive semi-definite random matrices, we have already shown that ${\displaystyle \sum_{k=1}^{n} H_k^\alpha = I_{(N+1)^d}.}$ Consequently, we have $\mu_{\min}=\mu_{\max}=1.$ Moreover, in this case, the constant $B$ is given by \eqref{Eq3.6}. Hence, by applying \eqref{Eq3.7} and \eqref{Eq3.8}  with $0< \delta <1,$ $D=(N+1)^d,$ $B$ as given by \eqref{Eq3.6}
and by combining the obtained both inequalities, one gets the desired estimate \eqref{kkappa2G}.\\

\begin{remark}
Note that instead of using Chernoff theorem and its consequence to get with high probability an estimate for the deviations of the largest and smallest eigenvalues of the random positive definite matrix $\mathbf A$ from those of $\mathbb E(\mathbf A),$ one can use the McDiarmid's concentration inequality, together with the techniques developed in \cite{Bonami-Karoui}. 
\end{remark}

\begin{remark}
It is easy to see that for $0<\delta <1,$ the estimate \eqref{kkappa2G}
implies that  our estimator is stable with high probability whenever the   number of sampling points $n$  satisfies
\begin{equation}
\label{Eqq1.5}
 n\geq \frac{3}{\delta^2}  \log\big(2(N+1)^d\big)\Big( B_\alpha (N+1)^{2\alpha+2} \Big)^{d}.
\end{equation}
 Moreover, from the previous inequality, one concludes that the special value of ${\displaystyle \alpha=-\frac{1}{2}}$ (corresponding to the tensor product of Chebyshev polynomials) is a convenient choice. 
 In fact, this choice  ensures 
the stability of the estimator $\widehat f^\alpha_N$ with the smaller  values of $n.$
\end{remark}

A second main result of this work is the following theorem that provides us with a weighted $L^2-$error of our estimator $\widehat f^\alpha_N.$

\begin{theorem}
\label{Thm2}
For fixed real number $\alpha\geq -\frac{1}{2}$ and a positive integer $N\geq 1,$ let $f\in L^2(I^d,\pmb \omega_{\alpha})$ be as given by \eqref{model1}. We assume that $ \|\pi_N f\|_\alpha >0$ and
\begin{equation}
\label{Eq3.17}
\esssup_{\pmb x\in I^d} |\pi_Nf(\pmb x)| \leq M_N,\quad  \quad \esssup_{\pmb x\in I^d} |f(\pmb x)| =\|f\|_\infty<+\infty,\quad  \mathbb P \Big( |\varepsilon_i| > M_\varepsilon \Big) = p_\epsilon,\quad 0< p_\epsilon \ll 1.
\end{equation}
 Then, under  the hypotheses of Theorem 1, for any $0<\delta < \|\pi_N(f)\|_\alpha ,$ we have with probability at least ${\displaystyle (1-p_\epsilon)^n-\exp\Big(\frac{-2n \delta^2}{\gamma^2_{\alpha,d} M_N^4}\Big)-\exp\left(\frac{-n \delta^2}{\gamma_{\alpha,d}^2 \max\Big(\|f-\pi_N f\|^4,M^4_\varepsilon\Big)}\right),}$
\begin{equation}
\label{Eq3.18}
 \|f-\widehat f^\alpha_N\|_\alpha\leq \|f-\pi_N(f)\|_\alpha+\sqrt{\kappa_2(G^\alpha_{d,N})}
 \Big(\|f-\pi_N(f)\|^2_\infty+\sigma^2+\delta\Big)^{1/2}\frac{1}{\sqrt{1-\frac{\delta}{\|\pi_N(f)\|_\alpha}}},
\end{equation}
where ${\displaystyle  \gamma_{\alpha,d}=\big(\beta(\alpha+1,\alpha+1)\big)^d}.$
\end{theorem}

\noindent
{\bf Proof:}  Let $f\in L^2(I^d,\pmb \omega_{\alpha}),$ then for an integer $N\geq 1,$ let $\pi_N(f)$ denote the orthogonal projection of $f$ over $\mathcal H_N=\mbox{Span}\{\Psi_m^\alpha,\, m=0,\ldots, N^d\}.$ That is 
\begin{equation}
\label{Eq3.9}
\pi_N(f)(\pmb x)=\sum_{m=0}^{N^d} <f,\Psi_m^\alpha>_\alpha \Psi_m^\alpha(\pmb x),\qquad <f,\Psi_m^\alpha>_\alpha=\int_{I^d} f(\pmb x) \Psi_m^\alpha(\pmb x) \pmb \omega_{\alpha}(\pmb x) d\pmb x.
\end{equation}
Let $c_m(f)=<f,\Psi_m^\alpha>_\alpha$ and let $\mathbf C_N=\big[c_m(f)\big]^T_{0\leq m\leq N^d}.$ From the uniqueness of the expansion coefficients of $f$ with respect to the orthonormal basis $\{\Psi_m^\alpha,\, m\in \mathbb N_0\}$ and by substituting $\pmb x$ with  the sampling points $\mathbf X_i,$ one concludes that the finite length 
expansion coefficients vector $\mathbf C_N$  satisfies the identity
\begin{equation}\label{Eq3.10}
F^\alpha_{d,N} \mathbf C_N =\frac{\sqrt{\gamma_{\alpha,d}}}{\sqrt{n}} \Big[\pi_N f(\mathbf X_1),\ldots,\pi_N f(\mathbf X_{n})\Big]^T,
\end{equation}
where the matrix $F^\alpha_{d,N}$ is as given by \eqref{matrixG}. On the other hand, our multivariate nonparametric estimator $\widehat f^\alpha_N$ is given by
\begin{equation}\label{Eq3.11}
F^\alpha_{d,N}  \widehat{\pmb  C}_N= \frac{\sqrt{\gamma_{\alpha,d}}}{\sqrt{n}} \Big[Y_i\Big]_{1\leq i\leq n}=\frac{\sqrt{\gamma_{\alpha,d}}}{\sqrt{n}} \Big[f(\mathbf X_i)+\varepsilon_i\Big]_{1\leq i\leq n}.
\end{equation}
Hence, by comparing \eqref{Eq3.10} and \eqref{Eq3.11}, one concludes that the least square norm solution of system \eqref{Eq3.11} can be viewed as a perturbation of the least square solution of system \eqref{Eq3.10}. More precisely, we have
$$F^\alpha_{d,N}  \widehat{\pmb  C}_N=F^\alpha_{d,N}{\mathbf  C_N}+
F^\alpha_{d,N}  (\widehat{\pmb  C}_N-\mathbf C_N),$$
where 
$$F^\alpha_{d,N}  (\widehat{\pmb  C}_N-\mathbf C_N)=\frac{\sqrt{\gamma_{\alpha,d}}}{\sqrt{n}} \Big[(f-\pi_N (f))(\mathbf X_i)+\varepsilon_i\Big]_{1\leq i\leq n}.$$
From the classical perturbation theory of least square norm solution of perturbed 
overdetermined system of linear equations, see for example \cite{Zi}, one has
$$ \frac{\|\mathbf C_N-\widehat{\pmb C}_N\|_{\ell_2}^2}{\|\mathbf C_N\|_{\ell_2}^2}\leq  \kappa_2(G^\alpha_{d,N}) \frac{\|F^\alpha_{d,N}(\mathbf C_N-\widehat{\pmb C}_N)\|_{\ell_2}^2}{\|F^\alpha_{d,N} \mathbf C_N\|_{\ell_2}^2}.$$
That is
\begin{equation}
\label{Eq3.12}
\|\mathbf C_N-\widehat{\pmb C}_N\|_{\ell_2}^2 \leq  \kappa_2(G^\alpha_{d,N}) 
\|F^\alpha_{d,N}(\mathbf C_N-\widehat{\pmb C}_N)\|_{\ell_2}^2 \frac{\|\mathbf C_N\|_{\ell_2}^2}{\|F^\alpha_{d,N} \mathbf C_N\|_{\ell_2}^2}.
\end{equation}
Next, since $\mathbb E \big[\gamma_{\alpha,d} \big(\pi_N f(\mathbf X_i)\big)^2\big]= \|\pi_N f\|^2_\alpha,$ then we have $$\mathbb E \Big[ \frac{\gamma_{\alpha,d}}{n}\sum_{i=1}^{n}\big(\pi_N f(\mathbf X_i)\big)^2\Big]= \|\pi_N f\|^2_\alpha =\|\mathbf C_N\|_{\ell_2}^2.$$ The last equality is a consequence of Parseval's equality.
Assume that ${\displaystyle \esssup_{\pmb x\in I^d} |\pi_N f(\pmb x)|\leq M_N}$, then by using Hoeffding's inequality, for any $\delta >0,$ we have 
$$\mathbb P \left(\frac{\gamma_{\alpha,d}}{n}\sum_{i=1}^{n}\big(\pi_N f(\mathbf X_i)\big)^2-\mathbb E\Big[ \frac{\gamma_{\alpha,d}}{n}\sum_{i=1}^{n}\big(\pi_N f(\mathbf X_i)\big)^2\Big]\geq \delta\right) \leq \exp\Big(\frac{-2n \delta^2}{\gamma^2_{\alpha,d} M_N^4}\Big).$$
That is 
\begin{equation}
\label{Eq3.13}
\| F^\alpha_{d,N} \mathbf C_N\|^2_{\ell_2} \geq \|\mathbf C_N\|^2_{\ell_2}-\delta,
\end{equation}
with probability at least ${\displaystyle 1-\exp\Big(\frac{-2n \delta^2}{\gamma^2_{\alpha,d} M_N^4}\Big).}$ On the other hand, we have 
\begin{eqnarray*}
\| F^\alpha_{d,N} \big(\widehat{\pmb C}_N-\mathbf C_N\big)\|^2_{\ell_2}&=&
\frac{\gamma_{\alpha,d}}{n}\sum_{i=1}^{n}\Big((f-\pi_N f)(\mathbf X_i)+\varepsilon_i\Big)^2. 
\end{eqnarray*}
Since $\mathbb E[\varepsilon_i]=0$ and since the $\mathbf X_i$ and $\varepsilon_i$ are independent, then it is easy to see that 
\begin{equation}
\label{Eq3.13-1}
\mathbb E\Big[\| F^\alpha_{d,N} \big(\widehat{\pmb C}_N-\mathbf C_N\big)\|^2_{\ell_2}\Big]=\| f-\pi_N f\|^2_\alpha +\sigma^2.
\end{equation}
Next, consider the tensor product set   ${\displaystyle \mathcal D= \prod_{i=1}^n I^d\times \prod_{i=1}^n \mathcal A,}$ where $\mathcal A$ is the support of the $\varepsilon_i$ which is a subset of $\mathbb R,$ that might be unbounded. Let $h$ be the real valued function $h$ defined on $\mathcal D$ by 
$$h_{\pmb{\varepsilon}}(\pmb x_1,\ldots,\pmb x_n)=\frac{\gamma_{\alpha,d}}{n}\sum_{i=1}^{n}\Big((f-\pi_N f)(\pmb x_i)+\epsilon_i\Big)^2. 
$$
Note that if $(\mathbf X_1,\ldots,\mathbf X_n,\epsilon_1,\ldots,\epsilon_n), (\mathbf X'_1,\ldots,\mathbf X'_n,\epsilon'_1,\ldots,\epsilon'_n)\in \mathcal D$ differ only in the $k-$th coordinate, then the following bounded difference condition holds with high probability $1-p_\varepsilon,$ 
\begin{eqnarray}\label{Eq3.13-2}
\Big| h_{\pmb{\varepsilon}}(\mathbf X_1,\ldots,\mathbf X_n)-h_{\pmb{\varepsilon'}}(\mathbf X'_1,\ldots,\mathbf X'_n)\Big|&\leq&\frac{\gamma_{\alpha,d}}{n} \max\Big(\|f-\pi_N f\|^2_\infty, M^2_\varepsilon\Big)
\end{eqnarray}
From \cite{Combes}, for any $\delta > 0$  and due to the tensor product structure of the set 
${\displaystyle \mathcal D_1=\prod_{i=1}^n I^d\times \prod_{i=1}^n [-M_\varepsilon,M_\varepsilon],}$ one has
\begin{equation}\label{Eq3.13-3}
\mathbb P\Big(h_{\pmb{\varepsilon}}(\mathbf X_1,\ldots,\mathbf X_n)-\big(\|f-\pi_N f\|_\alpha +\sigma^2\big) \geq \delta,\quad (\mathbf X_1,\ldots,\mathbf X_n,\varepsilon_1,\ldots,\varepsilon_n)\not\in \mathcal D_1\Big) \leq 1-(1-p_\varepsilon)^n.
\end{equation}
Moreover, on $\mathcal D_1,$ McDiarmid's inequality gives us
\begin{equation}\label{Eq3.13-4}
\mathbb P\Big(h_{\pmb{\varepsilon}}(\mathbf X_1,\ldots,\mathbf X_n)-\big(\|f-\pi_N f\|_\alpha +\sigma^2\big) \geq \delta,\, (\mathbf X_1,\ldots,\mathbf X_n,\varepsilon_1,\ldots,\varepsilon_n)\in \mathcal D_1\Big) \leq  \exp\left( \frac{- n\delta^2}{\gamma^2_{\alpha,d} \max\Big(\|f-\pi_N f\|^4_\infty, M^4_\varepsilon\Big)}\right).
\end{equation}
By combining \eqref{Eq3.13-3} and \eqref{Eq3.13-4}, one gets 
\begin{equation}
\label{Eq3.14}
\| F^\alpha_{d,N} \big(\widehat{\pmb C}_N-\mathbf C_N\big)\|^2_{\ell_2} \leq  \| f-\pi_N f\|^2_{\alpha}+\sigma^2+ \delta
\end{equation}
with probability at least ${\displaystyle (1-p_\varepsilon)^n-\exp\left( \frac{- n\delta^2}{\gamma^2_{\alpha,d} \max\Big(\|f-\pi_N f\|^4_\infty, M^4_\varepsilon\Big)}\right).}$ By combining \eqref{Eq3.12}, \eqref{Eq3.13} and \eqref{Eq3.14},
one concludes that for any $0<\delta < \|\pi_N f\|_\alpha^2,$ we have 
\begin{equation}
\label{Eq3.15}
 \|\widehat{\pmb C}_N-\mathbf C_N\big)\|^2_{\ell_2} \leq  \kappa_2 (G^\alpha_{d,N}) \Big(\|f-\pi_N f\|_\alpha^2+\sigma^2+\delta\Big)
\frac{\|\pi_N f\|_\alpha^2}{\|\pi_N f\|_\alpha^2-\delta},
\end{equation}
with probability at least ${\displaystyle (1-p_\varepsilon)^n-\exp\left( \frac{- n\delta^2}{\gamma^2_{\alpha,d} \max\Big(\|f-\pi_N f\|^4_\infty, M^4_\varepsilon\Big)}\right)-\exp\Big(\frac{-2n \delta^2}{\gamma^2_{\alpha,d} M_N^4}\Big).}$
Finally, we note that by Parseval's equality, we have 
$\|\widehat{\pmb C}_N-\mathbf C_N\big)\|^2_{\ell_2}=\|\pi_N f -\widehat f^\alpha_N\|^2_\alpha.$ Hence, to conclude the proof, it suffices to combine  \eqref{Eq3.15} with  the inequality
$${\displaystyle \|f-\widehat f^\alpha_N\|_\alpha\leq \|f-\pi_N f\|_\alpha+\|\pi_N f-\widehat f^\alpha_N\|_\alpha}.$$

\begin{remark}
It is interesting to note that in practice, the probability $p_\epsilon$ is fairly small for moderate values of the truncation bound $M_\varepsilon.$ 
For example, for the largely used Gaussian white noise model with variance $\sigma^2,$ for  any fixed $K>0$ and for any $i\in \mathbb N,$  we have $|\varepsilon_i|\geq K \sigma$ with probability at most ${\displaystyle  \mbox{erf}\Big(\frac{K}{\sqrt{2}}\Big)\approx  \frac{e^{-K^2/2}}{K \sqrt{\pi/2}}}.$ This last quantity is very close to $0$ even for  small positive values of $K.$ 
\end{remark}

Note that unless  ${\displaystyle \sigma^2=0}$ in the previous theorem  vanishes (that is the very special case  of noiseless nonparametric regression model), the integrated error bound \eqref{Eq3.18} has the drawback to lack of a convergence rate to zero, in terms of the parameters $n ,N.$ To overcome this problem,   we give in the sequel
an estimate of the $L_2$-risk error of a truncated version of our estimator $\widehat f^\alpha_N.$ The techniques used to get this $L_2$-risk are similar to  those used in \cite{BenSaber-Karoui}
in the univariate case. We assume that the regression function $f$ is almost everywhere bounded by a constant $M_f$, that is
$$|f(\pmb x)|\leq M_f,\quad \mbox{a.e.} \quad \pmb x\in  I^d.$$ 
Let $\widehat F_{N}$ be the truncated version of the estimate $\widehat f^\alpha_N$ given by
\begin{equation} \label{1'}
\widehat F_{N}(\pmb x)=\mbox{sign}(\widehat f^\alpha_N(\pmb x))\min(M_f, |\widehat f^\alpha_N(\pmb x)|),\quad \pmb x\in  I^d.
\end{equation} 
Under the usual assumption that the $\varepsilon_i$ are the  $n$ i.i.d.  centered  random  noises with variance $\sigma^2$, we have the following theorem that provides us with an estimate of the $L_2$-risk error of the estimator $\widehat F_N$. The proof of this theorem is partly inspired from the techniques developed in \cite{Cohen}.
\begin{theorem}
 Let $\alpha \geq -\frac{1}{2}$ and let  $0<\delta<1.$ Then, under the previous notations and hypotheses, we have
\begin{eqnarray}\label{L2risk}
\mathbb E\Big[\| f-\widehat F_N \|_\alpha^2\Big]&\leq&\frac{(N+1)^d}{n(1-\delta)^2}\Bigg(\sigma^2+\big(\eta_\alpha^2(N+1)^{2\alpha+1}\big)^d\| f-\pi_Nf \|_\alpha^2\Bigg)+\| f-\pi_Nf \|_\alpha^2\nonumber\\
&&\qquad +
 4M_f^2(N+1)^d\big(\beta(\alpha+1,\alpha+1)\big)^{d} \exp\Big(-\frac{n \delta^2}{2 \big(B_\alpha (N+1)^{2\alpha +2}\big)^d}\Big),
\end{eqnarray}	
where $B_\alpha$ is a constant depending only on $\alpha.$
\end{theorem}

\begin{proof}
 Recall that from \eqref{Eq3.7}, we have for any $\delta \in(0,1]$
$$\mathbb P\Big(\lambda_{\min}(G^\alpha_{d,N})\geq 1-\delta\Big)\geq 1- (N+1)^d \exp\Big(-\frac{n \delta^2}{2 \big(B_\alpha (N+1)^{2\alpha +2}\big)^d}\Big)$$ where $B_{\alpha} \leq 1$ is a constant depending only on $\alpha.$ As it is done in \cite{Cohen}, see also \cite{BenSaber-Karoui}, let $\pmb \Omega_+$ and $\pmb \Omega_-$ be the subsets of  $(I^d)^n$ given by  all possible draw $(\mathbf X_1,\cdots, \mathbf X_n)$ with $\lambda_{\min}(G^\alpha_{d,N})\geq 1-\delta$ and  $\lambda_{\min}(G^\alpha_{d,N})< 1-\delta,$ respectively. Let $d\pmb \rho_n$ be the probability measure on $\mathcal U^n$, given by the tensor product 
$$ d \pmb \rho_n= \prod_{i=1}^n d h_{\alpha+1}(\pmb x_i),$$ 
where $ h_{\alpha+1}(\pmb x_i)$ is as given by \eqref{Eq1.4}.
Then, we have
\begin{equation}
\int_{\pmb \Omega_{-}} d \pmb \rho_n=\mathbb P\big\{(\mathbf X_1,\cdots, \mathbf X_n)\in\mathcal U^n; \lambda_{\min}(G^\alpha_{d,N})< 1-\delta \big\}\leq (N+1)^d \exp\Big(-\delta^2\frac{n}{2 \big(B_\alpha (N+1)^{2\alpha +2}\big)^d}\Big) .
\end{equation}
Next, by using \eqref{1'},  the truncated estimator $\widehat F_N$ satisfies 
\begin{equation}\label{2'}
|f(\pmb x)-\widehat F_N(\pmb x)|\leq |f(\pmb x)-\widehat f^\alpha_N(\pmb x)| \leq|f(\pmb x)|+|\widehat F_N(\pmb x)|\leq 2M_f, \quad \forall \,\pmb x \in I^d.
\end{equation} 
Hence, we have 
\begin{equation}\label{4'}
\mathbb E\big(\| f-\widehat F_N \|_\alpha^2\big)=\int_{\pmb \Omega_{+}}\| f-\widehat F_N \|_\alpha^2 d \pmb \rho_n+\int_{\pmb \Omega_{-}}\| f-\widehat F_N \|_\alpha^2 d 
\pmb \rho_n.
\end{equation}
By using \eqref{2'}, one gets 
\begin{equation}\label{5'}
\int_{\pmb \Omega_{-}}\| f-\widehat F_N \|_\alpha^2 d \pmb \rho_n\leq 4M_f^2(N+1)^d\big(\beta(\alpha+1,\alpha+1)\big)^{d} \exp\Big(-\delta^2\frac{n}{2 \big(B_\alpha (N+1)^{2\alpha +2}\big)^d}\Big).
\end{equation}
On the other hand, from \eqref{2'}, we have 
\begin{eqnarray*}
	\int_{\pmb \Omega_{+}}\| f-\widehat F_N \|_\alpha^2 d \pmb \rho_n&\leq&\int_{\pmb \Omega_{+}}\| f-\widehat f^\alpha_N \|_\alpha^2 d \pmb \rho_n\\
	&\leq &\int_{\pmb \Omega_{+}}\| f-\pi_Nf \|_\alpha^2 d \pmb \rho_n +\int_{\pmb \Omega_{+}}\| \pi_Nf-\widehat f^\alpha_N \|_\alpha^2 d \pmb \rho_n.
\end{eqnarray*}
Note that by  Parseval's equality, we have  on $\pmb \Omega_{+},$
\begin{eqnarray*}
	\| \pi_Nf-\widehat f^\alpha_N\|_\alpha^2&=& \|\widehat{\pmb C}_N-\mathbf C_N\|^2_{\ell_2}\leq \|\big(G^\alpha_{d,N}\big)^{-1}\|_2^2\,\, \|\big(F^\alpha_{d,N}\big)^T\mathbf {\Delta P}\|_{\ell_2}^2\\
	&\leq &\frac{1}{(1-\delta)^2}\|\big(F^\alpha_{d,N}\big)^T\mathbf 
	{\Delta P}\|_{\ell_2}^2,
\end{eqnarray*}
where $\mathbf {\Delta P}=\frac{1}{\sqrt{n}} \Big[(f-\pi_N (f))(\mathbf X_i)+ \varepsilon_i\Big]_{1\leq i\leq n}.$
This last inequality implies 
 $$\displaystyle\mathbb{E}\Big[\| \pi_Nf-\widehat f^\alpha_N \|_\alpha^2\Big]\leq\frac{1}{(1-\delta)^2}\displaystyle\mathbb E\Big[\|F^\alpha_{d,N}\big)^T\mathbf {\Delta P}\|_{\ell_2}^2\Big].$$
Straightforward computation gives us $$\|\big(F^\alpha_{d,N}\big)^T\mathbf {\Delta P}\|_{\ell_2}^2=\frac{\big(\beta(\alpha+1,\alpha+1)\big)^{d}}{n^2} \sum_{\pmb k\in [[0,N]]^d}^{}\sum_{j,l=1}^{n}\Phi_{\pmb k}^\alpha(\mathbf X_j)(\theta_N(\mathbf X_j)+\varepsilon_j)\Phi_{\pmb k}^\alpha(\mathbf X_l)(\theta_N(\mathbf X_l)+\varepsilon_l)$$
where $\theta_N(\cdot)=(f-\pi_Nf)(\cdot)\quad\bot\quad\Phi_{\pmb k}^\alpha(\cdot), \quad \forall\; \pmb k \in [[0,N]]^d.$
Since the $\varepsilon_j$'s are independent of the $\mathbf X_j$'s, and since $\mathbb E\big[\varepsilon_j\big]=0,$ $\mathbb E\big[\varepsilon_j^2\big]=\sigma^2$, then we have
\begin{eqnarray}\label{5'}
\mathbb E\Big[\|\big(F^\alpha_{d,N}\big)^T\mathbf {\Delta P}\|_{\ell_2}^2\Big]&=&\frac{\big(\beta(\alpha+1,\alpha+1)\big)^{d}}{n^2}\sum_{\pmb k \in [[0,N]]^d}\sum_{j=1}^{n}\mathbb E\Big[\pmb\varepsilon_j^2\big(\Phi_{\pmb k}^\alpha(\mathbf X_j)\big)^2\Big] \nonumber \\
&&\qquad+\frac{\big(\beta(\alpha+1,\alpha+1)\big)^{d}}{n^2}\sum_{\pmb k \in [[0,N]]^d}\mathbb E\Big[\sum_{j=1}^{n}\big(\Phi_{\pmb k}^\alpha(\mathbf X_j)\big)^2\big( \theta_N(\mathbf X_j)\big)^2\Big].
\end{eqnarray} 
Since $${\displaystyle \mathbb E\Big[\varepsilon_j^2\big(\Phi_{\pmb k}^\alpha(\mathbf X_j)\big)^2\Big]=\mathbb E\big[\varepsilon_j^2\big]\mathbb E\Big[\big(\Phi_{\pmb k}^\alpha(\mathbf X_j)\big)^2\Big]=\sigma^2 \frac{1}{\Big(\beta(\alpha+1,\alpha+1)\Big)^d}}$$ and since $${\displaystyle \mathbb E\Big[( \theta_N(\mathbf X_j))^2\Big]=\frac{1}{\Big(\beta(\alpha+1,\alpha+1)\Big)^d}\| \theta_N\|_\alpha^2},$$ then by using the fact  ${\displaystyle \sum_{\pmb k \in [[0,N]]^d}^{}(\Phi_{\pmb k}^\alpha(\mathbf X_j))^2\leq \Big(\eta_\alpha^2(N+1)^{2\alpha+2}\Big)^d },$  one gets 
\begin{equation}
\frac{\big(\beta(\alpha+1,\alpha+1)\big)^{d}}{n^2}\sum_{\pmb k \in [[0,N]]^d}^{}\displaystyle\mathbb E\Big[\sum_{j=1}^{n}\big(\Phi_{\pmb k}^\alpha(\mathbf X_j)\big)^2\big( \theta_N(\mathbf X_j)\big)^2\Big]\leq \frac{\Big(\eta_\alpha^2(N+1)^{2\alpha+2}\Big)^d}{n}\| \theta_N\|_\alpha^2.
\end{equation}
To conclude for the proof of the theorem, it suffices to combine   \eqref{2'}--\eqref{5'} and get \eqref{L2risk}.

\end{proof}

Next, we show that if $f$ belongs to the functional space  of $2-$norm isotropic Soblev space $H^s(I^d),$ with an appropriate $ s>0,$
then $f$ satisfies condition \eqref{Eq3.17}. Moreover, for such a function, one also gets an estimate for the quantity $\|f-\pi_N f\|_\alpha$ given in 
\eqref{Eq3.18}. The $2-$norm  isotropic Soblev space $H^s(I^d)$ is given by, see for example \cite{Cobos}
\begin{equation}
\label{Eq3.19}
H^s(I^d)=\Big\{ f\in L^2(I^d),\, \sum_{\pmb k\in \mathbb Z^d} \Big(1+\sum_{j=1}^d |k_j|^2\Big)^{s} \big|<e^{2i\pi \pmb k\cdot}, f>\big|^2<+\infty \Big\},
\end{equation}
where, ${\displaystyle <e^{2i\pi \pmb k\cdot}, f>=\int_{I^d} f(\pmb x) e^{-2i\pi \pmb k \cdot \pmb x} d\pmb x.}$ Also, we recall that if $f\in L^2(I^d),$ then by using the notation  $\pmb m=(m_1,\ldots,m_d)\in \mathbb Z^d$,  ${\displaystyle \|\pmb m\|_\infty=\max_{1\leq j\leq d} |m_j|},$
the projection $\pi_N f $ is given by 
$$\pi_N f (\pmb x)=\sum_{\|\pmb m\|_\infty\leq N} C_{\pmb m} \Phi_m^\alpha(\pmb x),\quad C_{\pmb m}=<f,\Phi_m^\alpha(\pmb x)>_\alpha,$$
where $\Phi_m^\alpha(\pmb x)$ is as given by \eqref{Eq1.1}. In the sequel, we use the notation $\lesssim_{\alpha,d}$ to say that the inequality holds up to a constant depending only on $\alpha$ and $d$. This last constant is generic and may change from one line to another line. 

\begin{theorem}\label{Thm3} Under the previous notations, let $s>0$ and
$\alpha\geq -\frac{1}{2}.$ For any integer   $N\geq 2$  satisfying ${\displaystyle \frac{N}{\log N}\geq \frac{1}{\log 2} \big(s+d+\frac{1}{2}\big)}$ and  for $f\in H^{s+\frac{d}{2}}(I^d),$ we have 
\begin{equation}\label{Eq3.20}
\|f-\pi_N f\|_\alpha \lesssim_{\alpha,d}  \sqrt{\frac{d}{2s}}\,\, N^{-s}.
\end{equation}
Moreover if  $s> d(\alpha+1),$ then  we have
\begin{equation}\label{Eq3.21}
\|f-\pi_N f\|_\infty \lesssim_{\alpha,d} \frac{d}{s-d(\alpha+1)} N^{-s+d(\alpha+1)}.
\end{equation} 
\end{theorem}

\noindent
{\bf Proof:} Since the family of multivariate trigonometric exponentials  $\{e^{2i\pi \pmb k\cdot \pmb x},\, \pmb k\in \mathbb Z^d\}$ is an orthonormal basis of $L^2(I^d),$ then we have 
$$ f(\pmb x)= \sum_{\pmb k\in \mathbb Z^d} a_{\pmb k}(f) e^{2i\pi \pmb k\cdot \pmb x},\quad x\in I^d,\quad  a_{\pmb k}(f)=<f(\cdot), e^{-2i\pi \pmb k\cdot}>.$$ It is not hard to see that if $\pmb m\in \mathbb N_0^d,$ and 
$C_{\pmb m}(f)=<f,\Phi_{\pmb m}^\alpha>_\alpha,$ then for  $f\in H^s(I^d)$ with $s>d$ we have
$$C_{\pmb m}(f)=\sum_{\pmb k\in \mathbb Z^d} d_{\pmb k,\pmb m}  a_{\pmb k}(f),\quad d_{\pmb k,\pmb m}= <\Phi_{\pmb m}^\alpha(\pmb x),e^{2i\pi \pmb k\cdot \pmb x}>_\alpha.$$
Taking into account that $$ \Phi_{\pmb m}^\alpha(\pmb x)= \prod_{j=1}^d \wJ_{m_j}(x_j),\quad \pmb x=(x_1,\ldots,x_d),\quad \pmb m=(m_1,\ldots,m_d)\in \mathbb N_0^d,$$ one gets 
\begin{equation}
\label{Eq4.0}
d_{\pmb k,\pmb m}= \prod_{j=1}^d < e^{2i\pi k_j x_j},\wJ_{m_j}(x_j)>_\alpha= \prod_{j=1}^d d_{k_j,m_j},\quad d_{k_j,m_j}=\int_I e^{2i\pi k_j x_j}\wJ_{m_j}(x_j) \omega_{\alpha}(x_j)\, dx_j.
\end{equation}
On the other hand, it is known that, see for example \cite{NIST}
\begin{equation}\label{Eq4.1}
 \int_{-1}^{1} e^{ixy} \overline{P_m}^{(\alpha , \alpha)}(y) \overline{\omega}_{\alpha}(y)\,  dy= i^m \sqrt{\pi}\sqrt{2m+2\alpha+1} 
\sqrt{\frac{\Gamma(m+2\alpha+1)}{\Gamma(m+1)}}\frac{J_{m+\alpha+1/2}(x)}{x^{\alpha+1/2}},\quad x\in \mathbb R.
 \end{equation}
 Here, the $\overline{P_m}^{(\alpha , \alpha)}$ are the orthonormal Jacobi on $[-1,1],$ with $\overline{\omega}_{\alpha}(y)=(1-y^2)^\alpha,$   $J_{a}$ is the Bessel function of the first kind and order $a>-1$
 and $\Gamma(x)$ is the usual Gamma function.  The Gamma and Bessel functions  $\Gamma(\cdot)$ and $J_a(\cdot)$  satisfy the following useful inequalities that can be found in the literature,
 \begin{equation}\label{Eq4.2}
 \sqrt{2e} \left(\frac{x+1/2}{e}\right)^{x+1/2}\leq \Gamma(x+1)\leq \sqrt{2\pi} \left(\frac{x+1/2}{e}\right)^{x+1/2},\quad x>0. 
 \end{equation}
 and
  \begin{equation}\label{Eq4.3}
  |J_\mu(x)| \leq \frac{|x|^\mu}{2^\mu \Gamma(\mu+1)},\quad \mu>-1,\quad x\in \mathbb R.
 \end{equation}
 By using \eqref{Eq4.1} and \eqref{Eq4.2}, one gets 
\begin{equation}\label{Eq4.4}
d_{k_j,m_j}= \frac{(-1)^{k_j}}{2} i^{m_j}\sqrt{\pi} \sqrt{2m_j+2\alpha+1}\sqrt{\frac{\Gamma(m_j+2\alpha+1)}{\Gamma(m_j+1)}} \frac{J_{m_j+\alpha+1/2}(\pi k_j)}{(\pi k_j)^{\alpha+1/2}}.
\end{equation} 
Note that since the function ${\displaystyle x\rightarrow \frac{J_{m+\alpha}(x)}{x^{\alpha}}}$ has same parity as $m,$ then one can only consider the case $k_j\geq 0$ in \eqref{Eq4.4}. The value of $d_{k_j,m_j}$ for $k_j<0$ is simply given by $d_{k_j,m_j}=(-1)^{m_j} d_{-k_j,m_j}.$ Hence, by using \eqref{Eq4.1}, \eqref{Eq4.2} and \eqref{Eq4.3} together with some straightforward computations, one gets the useful inequality 
\begin{equation}
\label{Eq4.5}|d_{k_j,m_j}| \lesssim_{\alpha,d} \sqrt{m_j} \left(\frac{e\pi |k_j|}{2 m_j+2\alpha}\right)^{m_j},\quad m_j\geq 1.
\end{equation}
Moreover, from Cauchy-Schwarz inequality, we also have $|d_{k_j,m_j}|\leq 1$
for any integers $k_j, m_j.$ Consequently, if 
${\displaystyle \|\pmb k\|_\infty \leq \frac{\|\pmb m\|_\infty}{e \pi},}$
then one concludes that 
\begin{equation}
\label{Eq4.6}
|d_{\pmb k,\pmb m}| \lesssim_{\alpha,d} \sqrt{\|\pmb m\|_\infty}\,\,  2^{-\|\pmb m\|_\infty},\quad \forall\, \, \|\pmb k\|_\infty \leq \frac{\|\pmb m\|_\infty}{e \pi}.
\end{equation}
Next, we write the multivariate Jacobi coefficient expansion of $f\in H^s(I^d),$ as follows
\begin{equation}\label{Eq4.7}
C_{\pmb m}(f)=\sum_{ \|\pmb k\|_\infty \leq \|\pmb m\|_\infty/{e \pi}} d_{\pmb k,\pmb m}  a_{\pmb k}(f)+\sum_{\pmb  \|\pmb k\|_\infty > {\|\pmb m\|_\infty}/{e \pi}} d_{\pmb k,\pmb m}  a_{\pmb k}(f)=S_1+S_2,
\end{equation}
where, $a_{\pmb k}(f)=<f(\cdot), e^{-2i\pi \pmb k\cdot}>.$ To bound $|S_1|,$ we first note that in $\mathbb Z^d,$ there exist at most ${\displaystyle \left[\frac{2}{e\pi}\|\pmb m\|_\infty+1\right]^d}$ different $\pmb k$ satisfying $\|\pmb k\|_\infty \leq \|\pmb m\|_\infty/{e \pi},$ where $\left[x\right]$ denotes the integer part of $x$. Moreover, from Bessel's inequality, we have 
${\displaystyle \sum_{ \|\pmb k\|_\infty \leq \|\pmb m\|_\infty/{e \pi}} |a_{\pmb k}(f)|^2\leq \|f\|_2^2}.$ Consequently by using \eqref{Eq4.6} and Cauchy-Schwarz inequality, one concludes that 
\begin{equation}\label{Eq4.8}
|S_1| \lesssim_{\alpha,d} \|\pmb m\|_\infty^{\frac{1+d}{2}}\,\,  2^{-\|\pmb m\|_\infty} \|f\|_2. 
\end{equation}
On the other hand, since $f\in H^{s+d/2}(I^d)$ and since by Bessel's inequality, we have 
$$ \sum_{\pmb  \|\pmb k\|_\infty > {\|\pmb m\|_\infty}/{e \pi}} |d_{\pmb k,\pmb m}|^2 \leq \|\Phi^\alpha_{\pmb m}\|^2_\alpha=1,$$ then a simple Cauchy-Schwarz inequality gives us
\begin{eqnarray}\label{Eq4.9}
|S_2|^2 &\leq& \sum_{\|\pmb k\|_\infty >\|\pmb m\|_\infty/{e \pi}}  |a_{\pmb k}(f)|^2 \leq \Big(\frac{e\pi}{\|\pmb m\|_\infty}\Big)^{2s+d} \cdot \sum_{ \|\pmb k\|_\infty > \|\pmb m\|_\infty/{e \pi}} \Big(1+\sum_{j=1}^d |k_j|^2\Big)^{s+d/2}  |a_{\pmb k}(f)|^2 \nonumber \\
&\leq & \|\pmb m\|_\infty^{-2s} (e\pi)^{2s} \|f\|^2_{H^s}.
\end{eqnarray}
By combining \eqref{Eq4.8} and \eqref{Eq4.9}, one concludes that 
\begin{equation}
\label{Eq4.10}
|C_{\pmb m}| \lesssim_{\alpha,d} \Big(\|\pmb m\|_\infty^{\frac{1+d}{2}}\,\,  2^{-\|\pmb m\|_\infty}+ \|\pmb m\|_\infty^{-s+d/2}\Big) (\|f\|_2+\|f\|_{H^s}).
\end{equation}
Note that $\|\pmb m\|_\infty^{\frac{1+d}{2}}\,\,  2^{-\|\pmb m\|_\infty} \leq \|\pmb m\|_\infty^{-s-d/2},$ whenever $\|\pmb m\|_\infty\geq N,$ with 
${\displaystyle \frac{N}{\log N} \geq \frac{1}{\log 2} \big(s+d+\frac{1}{2}\big).}$ Hence, in this case, \eqref{Eq4.10} is simply written as 
\begin{equation}
\label{Eq4.11}
|C_{\pmb m}| \lesssim_{\alpha,d} \|\pmb m\|_\infty^{-s-d/2} (\|f\|_2+\|f\|_{H^s}).
\end{equation}
Next, since 
\begin{equation}
\label{Eq4.12}
 f(x)-\pi_N f(x)=\sum_{\pmb m\in \mathbb Z^d,\, \|\pmb m\|_\infty \geq N+1} C_{\pmb m} \Phi^\alpha_{\pmb m}(\pmb x)=\sum_{n=N+1}^\infty \sum_{\|\pmb m\|_\infty=n}C_{\pmb m} \Phi^\alpha_{\pmb m}(\pmb x).
\end{equation}
Again, since there exist $ d (n+1)^{d-1}$ $d-$tuples $\pmb m\in \mathbb N^d$ with $\|\pmb m\|_\infty=n,$ then by using \eqref{Eq4.11} and by Parseval's equality applied to \eqref{Eq4.12} (which is due to the  orthonormality of the $\Phi^\alpha_{\pmb m}$ in $L^2(I^,\pmb \omega_{\alpha})$) , one gets 
\begin{eqnarray}
\label{Eq4.13}
\| f-\pi_N f\|_\alpha^2 &\lesssim_{\alpha,d}& d \Big(\sum_{n=N+1}^\infty (n+1)^{d-1} n^{-2s-d} \Big) (\|f\|_2+\|f\|_{H^s})^2 \nonumber\\
&\lesssim_{\alpha,d}&\frac{d}{2s}\,   N^{-2s}   (\|f\|_2+\|f\|_{H^s})^2.
\end{eqnarray}
This concludes the proof of inequality \eqref{Eq3.20}. 
Finally to prove \eqref{Eq3.21}, we recall the following known upper bound
for the Jacobi polynomials $\wJ_m$ with  $\alpha\geq -\frac{1}{2},$ see for example
$$ \sup_{x\in I} |\wJ_m(x)| \leq M_\alpha m^{\alpha+\frac{1}{2}},$$
for some constant $M_\alpha.$ Consequently, we have 
$$\sup_{\pmb x\in I^d} |\Phi^\alpha_{\pmb m}(\pmb x)| \lesssim_{\alpha,d} \|\pmb m\|_{\infty}^{d(\alpha+1/2)}.$$
Hence, by using the previous technique we have used to bound $\| f-\pi_N f\|_\alpha,$ one gets 

\begin{eqnarray*}
\sup_{\pmb x\in I^d} \| f(\pmb x)-\pi_N f(\pmb x)\|& \lesssim_{\alpha,d}&
d \sum_{n=N+1}^\infty (n+1)^{d-1} n^{-s-d/2} n^{d(\alpha+1/2)} \\
&\lesssim_{\alpha,d}& N^{-s+d(\alpha+1)} \frac{d}{s-d(\alpha+1)}.
\end{eqnarray*}
This concludes the proof of the Theorem.

As a consequence of the previous two theorems, we have the following corollary that provides us with a convergence rate for our truncated estimator $\widehat F_N,$ when the regression function belongs to an isotropic Sobolev space.

\begin{corollary} Let $s>0$ and let $\alpha\geq -\frac{1}{2}$ be such 
${\displaystyle s> d\Big(\alpha+\frac{1}{2}\Big).}$ Assume that the regression function $f$ belongs to 
an isotropic Sobolev space $H^{s+d/2}(I^d),$ then the convergence rate of the estimator $\widehat F_N$ is of order ${\displaystyle O\Big(n^{-2s/(2s+d)}\Big).}$
\end{corollary}

\begin{proof} Straightforward computations show that if ${\displaystyle s> d\Big(\alpha+\frac{1}{2}\Big),}$ then the first and the third quantity in the sum of the left hand side of \eqref{L2risk} are of order 
${\displaystyle O\Big(\frac{(N+1)^d}{n}\Big).}$ Hence, by using \eqref{Eq3.20} and \eqref{L2risk},  the fastest rate of convergence of the estimator $\widehat F_N$ is obtained when $N=O\big(n^{1/(2s+d)}\big).$
In this case, we have 
$$\mathbb E\Big[\| f-\widehat F_N \|_\alpha^2\Big]=O\Big(n^{-2s/(2s+d)}\Big).$$
\end{proof}

\begin{remark} It is interesting to note that the previous convergence rate of our estimator $\widehat F_N$ is too similar to the theoretical optimal convergence rate of  min-max nonparametric estimator. This last optimal convergence rate is given in \cite{Stone}, see also \cite{Bauer}. It is given in the case where the regression function $f$ belongs to the class $C-p$ functions, that is the set of $d-$variate functions of class $C^p,$ with their different $p-$th order partial derivatives being H\"older continuous. In this case, the optimal rate of convergence is of 
$ O\Big(n^{-2p/(2p+d)}\Big)$, as for kernel regression estimate.
\end{remark}

\section{Computational analysis}

In this section, we check the performance of our estimator  $\widehat f^\alpha_N(\cdot),$ by applying it to synthetic data as well as to real data.

\subsection{Numerical simulations on synthetic data}

In this paragraph, we  give three numerical examples that illustrate the results of this work. The first example is an illustration of the first main Theorem \ref{Mainthm1}, while the other two  examples illustrate the performance of our estimator when applied to synthetic data.\\

\noindent
{\bf Example 1:}  In this first example, we check numerically the important result given by Theorem \ref{Mainthm1}. To this end, we have considered the dimension $d=2,$ then we have computed an average for the true condition number $\kappa_2 (G^\alpha_{d,N})$ over  10 realizations and for different values of $\alpha=-0.5,\, 0.0,$ $N=5, 10$ and $n=n_1^2,\, n_1=60,80,100.$ The obtained numerical results are given by Table \ref{tableau1} and they are fairly coherent with the theoretical result of Theorem \ref{Mainthm1}. As it is stated by formula \eqref{kkappa2G} of Theorem \ref{Mainthm1}, the value of $\alpha=-\frac{1}{2}$ seems to be the optimal value that provides the smallest condition number for the random matrix $G^\alpha_{d,N}.$
Moreover, for a fixed value of the parameter $\alpha,$ a smaller  value of $N$
and/or a larger  value of $n,$ give us a smaller condition number 
$\kappa_2 (G^\alpha_{d,N}).$

\begin{center}
\begin{table}[h]
\vskip 0.2cm\hspace*{4.0cm}
\begin{tabular}{ccccccc} \hline
 $\alpha$   &$N$&$n_1$&$\kappa_2 (G^\alpha_{d,N})$        &   $N$  &$n_1$&$\kappa_2 (G^\alpha_{d,N})$\\   \hline
 $-0.5$& $5$ & $60$      &$ 7.45  $ & $10$ & $ 60$& $35.41  $\\
 $$    & $-$ & $80$      &$ 5.62  $ & $-$ & $ 80$& $19.97  $\\
 $$    & $-$ & $100$     &$ 3.94  $ & $-$ & $ 100$& $11.41  $\\
 $0.0$& $5$ & $60$       &$13.51   $ & $10$ & $ 60$& $6596.05  $\\
  $$  & $-$ & $80$      &$10.17   $ & $-$ & $ 80$& $369.43  $\\
  $$  & $-$ & $100$     &$7.57   $ & $-$ & $ 100$& $158.53  $\\

  \\ \hline
\end{tabular}
\caption{Illustration of Theroem \ref{Mainthm1}'s results with $n=n_1^2.$}
\label{tableau1}
\end{table}
\end{center}

\noindent
{\bf Example 2:} In this second example, we illustrate  the performance of our estimator $\widehat f^\alpha_N(\cdot),$ given by \eqref{Eq1.3}--\eqref{matrixG},
by applying it to data, generated by a   synthetic $2-$variate regression function, given by 
$$ f(u,v)=1+2 u-4 v+3 v^2-2 u v+3 u v^2-v^3+u^4+2 u^5+\sin(2 \pi u)-\cos(3 \pi v),\quad u,v \in [0,1].$$ 
For this purpose, we have considered the special case of $\alpha=-\frac{1}{2},$ $N=5, 10$ and $n=n_1^2,\, n_1=60,\, 80,\,  100.$ Then we considered the regression problem (\ref{model1}) with i.i.d. white gaussian noises associated to the standard deviation $\sigma=0.05$ and $\sigma=0.15.$ To assess the performance of the regression estimate, we have computed the mean squared error (MSE), which is given by 
$$ MSE = \frac{1}{n} \sum_{i=1}^n \left(f(\pmb X_i)-\widehat f^\alpha_N(\pmb X_i)\right)^2.$$
Here, the $\pmb X_i$ are i.i.d. bi-variate random sampling points with each of the two components following a Beta distribution on $[0,1]$ and associated with the parameter $\alpha=-0.5.$  From these numerical results and as it has been stated by Theorem \ref{Thm3}, for fixed values of the parameters $\alpha,\sigma, n$ the largest the value of $N,$ the smallest is the associated MSE. Also, it is interesting to note that according to Theorem \ref{Mainthm1}, our estimator $\widehat f^\alpha_N(\cdot)$ is surprisingly stable in the sense that it behaves well in the presence of data perturbation by a white noise with reasonable variance. It is also more less time consuming than the kernel regression estimate (cf  Table \ref{tableau2}).\\  

\medskip

\begin{table}[h]
\vskip 0.2cm\hspace*{1.0cm}
\begin{tabular}{ccccccccccc} \hline
 $N$   &$\sigma$&$n_1$&$OM$  &$Kernel$      &   $\sigma$  &$n_1$&$OM$ &$Kernel$  &$TOM$ & $TK$\\   \hline
 $10$& $0.05$ & $60$      &$2.62e-3   $  &$1.4e-3   $ & $0.15$ & $ 60$& $2.64e-2  $ & $\mathbf{6.0e-3}  $  &1.55 & 1.10\\
     & $-$ & $80$      &$2.60e-3 $        & $3.0e-3 $ & $-$ & $ 80$& $\mathbf{2.30e-2}$ & $2.4e-2$ & $2.51$ & 4.33\\
     & $-$ & $100$     &$2.44e-3  $ &$2.6e-3  $& $-$ & $ 100$& $\mathbf{2.25e-2} $ & $2.3e-2 $ &3.52 & 10.25\\
 $5$& $0.05$ & $60$      &$4.67e-3 $  &$1.4e-3 $ & $0.15$ & $ 60$& $2.47e-2  $ & $\mathbf{6.0e-3}  $ & 0.18 & 1.10\\
    & $-$ & $80$      &$6.40e-3  $ &$3.0e-3  $& $-$ & $ 80$& $2.70e-2  $ & $\mathbf{2.4e-2}  $ & 0.29 & 4.33\\
    & $-$ & $100$     &$5.11e-3   $ &$2.6e-3   $& $-$ & $ 100$& $2.66e-2  $ & $\mathbf{2.3e-2 } $  & 0.38 & 10.25\\

  \\ \hline
\end{tabular}
\caption{The mean squared error for  the proposed method and the nonparametric kernel method on the considered case of Example 2 with $n=n_1^2.$ Kernel denotes the mean squared errors of the kernel regression estimate while $OM$ is for the proposed method.  The best mean squared errors are in bold. The last two columns give the computing time (in minutes) of the proposed  (TOM) and for the kernel (TK) methods.}
\label{tableau2}
\end{table}
Moreover, we have pushed forward the previous numerical simulations associated to the previous synthetic example by comparing our proposed regression estimator to  other well known parametric and non parametric regression estimators. These estimators are given by Kernel (the kernel regression with optimal bandwidth selection by cross-validation); SVM (support vector machine) and RF (the Random Forest regression estimator).
For this second set of simulations, we have used $80 \%$ of the sample size $n=n_1^2$ for the construction of the estimators and the remaining $20 \%$ of the sample size are used to validate the estimators by computing  the resulting mean squared, the mean absolute errors as well as the coefficient of adjustment $R^2$. The obtained numerical results are given in Table 3. These results indicate that for larger $N$, sample size $n$ and error variance, the proposed method is competitive to the kernel method, which outperforms. In fact, our method is less time consuming than the kernel method as shown in Table \ref{tableau2} and outperforms compare to random forest and support vector machine methods.\\

\begin{table}[h]
		{\footnotesize{
	\begin{tabular}{llllll lll lll lll lll}
		\hline
		&     &       & \multicolumn{3}{l}{OM}& \multicolumn{3}{l}{Kernel}  & \multicolumn{3}{l}{SVM} & \multicolumn{3}{l}{RF}\\ \hline
		$N$                          & $\sigma$   & $n_1$    & $MSE$   &$R^2$  & $MAE$  \vline& $MSE$   &$R^2$  & $MAE$ \vline& $MSE$   &$R^2$  & $MAE$ \vline& $MSE$   &$R^2$  & $MAE$\\ \hline
		\multirow{3}{*}{5}
		& .05     &60  &1.36e-2 &9.96e-1	&8.31e-2	&	2.86e-3&	9.99e-1 & 4.3e-2	&	6.74e-2&	9.98e-1&	5.28e-2  &	3.92e-1&	9.43e-1&	3.12e-1	\\\
		&  $-$    &80  &8.35e-3 &9.96e-1	&7.31e-2	&	2.93e-3&	9.99e-1 & 4.3e-2	&	6.67e-2&	9.99e-1&	5.14e-2  &	3.92e-1&	9.39e-1&	3.10e-1\ \\
		& $-$     &100&6.39e-3 &9.98e-1	&6.44e-2	&	2.68e-3&	9.99e-1 & 4.1e-2	&	6.07e-2&	9.98e-1&	4.79e-2  &	3.76e-1&	9.47e-1&	3.11e-1 \\\\
		\multirow{3}{*}{10}
		& .05     &60  &5.35e-3 &9.98e-1&5.61e-2	&	2.86e-3&	9.99e-1 & 4.3e-2	&	6.74e-2&	9.98e-1&	5.28e-2  &	3.92e-1&	9.43e-1&	3.12e-1	\\\
		&  $-$  & 80&5.45e-3 &9.98e-1	&5.46e-2	&	2.93e-3&	9.99e-1 & 4.3e-2	&	6.67e-2&	9.99e-1&	5.14e-2  &	3.92e-1&	9.39e-1&	3.10e-1\ \\
		& $-$    & 100 &4.16e-3 &9.98e-1	&4.80e-2	&	2.68e-3&	9.99e-1 & 4.1e-2	&	6.07e-2&	9.98e-1&	4.79e-2  &	3.76e-1&	9.47e-1&	3.11e-1 \\\\
		\multirow{3}{*}{5}
		& .15     &60  &3.35e-2 &9.98e-1	&1.46e-1	&	2.32e-2&	9.99e-1 & 1.2e-1	&	1.67e-1&	9.98e-1&	1.35e-1  &	5.1e-1&	9.1e-1&	4.14e-1	\\\
		&  $-$  & 80&2.97e-2 &9.99e-1	&1.36e-1	&	2.46e-2&	9.99e-1 & 1.2e-1	&	1.71e-1&	9.87e-1&	1.39e-1  &	4.11e-1&	9.34e-1&	3.27e-1\ \\
		& $-$    & 100 &2.73e-2 &9.99e-1	&1.31e-1	&	2.32e-2&	9.99e-1 & 1.2e-1	&	1.69e-1&	9.99e-1&	1.34e-1  &	4.02e-1&	9.39e-1&	3.30e-1 \\\\
		\multirow{3}{*}{10}
		& .15     &60  &3.08e-2 &9.99e-1	&1.39e-2	&	2.32e-2&	9.99e-1 & 1.2e-1	&	1.67e-1&	9.98e-1&	1.35e-1  &	5.1e-1&	9.1e-1&	4.14e-1	\\\
		
		&  $-$  & 80&2.97e-2 &9.99e-1	&1.28e-1	&	2.46e-2&	9.99e-1 & 1.2e-1	&	1.71e-1&	9.87e-1&	1.39e-1  &	4.11e-1&	9.34e-1&	3.27e-1\ \\
		& $-$    & 100 &2.73e-2 &9.99e-1&1.24e-1	&	2.32e-2&	9.99e-1 & 1.2e-1	&	1.69e-1&	9.99e-1&	1.34e-1  &	4.02e-1&	9.39e-1&	3.30e-1 \\
		\\
		\hline
	\end{tabular}}}
		\centering
		\caption{Validation of the regression model estimation on a testing sample of size $80\%$ of the sample of size $n=n_1^2$ given in Example 2. The mean squared error (MSE), mean absolute error (MAE) and $R^2$ for the proposed method (OM), the nonparametric kernel (Kernel), Support Vector Machine (SVM) and Random Forest (RF) methods are given.}
		\label{pagre}
\end{table}

\noindent
{\bf Example 3:} In this  example, we illustrate  the performance of our estimator 
by applying it to a classification problem, generated by the   synthetic $2-$variate sample $(X_i,Y_i)_{i=1}^n,\, n=n_1^2$, $\mathbf X_i=(X_{i,1},X_{i,2})^\top$ of Example 2. The classification data is generated by the rule   $Y_i=1$ if $f(X_{i,1},X_{i,2})>c$, $Y_i=0$ otherwise, $c$ is the taken as the mean of the $f(X_{i,1},X_{i,2})$. We consider as in Example 2, a gaussian white noise associated to the two values of  $\sigma=0.05$ and $\sigma=0.15.$ Then, we have constructed our estimator $\widehat f^\alpha_N(\cdot),$  with $\alpha=-0.5$ $N=5, 10$ by using $80\%$ of the samples data with size $n=n_1^2$ for the different values of $n_1=60, 80, 100.$ The remaining  $20\%$ of the data are used for testing the classification
performance. Note that we have used the standard classification rule $\widehat Y_i=0$ if  $\widehat f^\alpha_N(X_i)\leq 0.5,$ otherwise $\widehat Y_i=1.$ Moreover, we have compared our proposed classifier (OM)  with other three estimators frequently used in the literature for classification purposes. These estimators are the LDA (Linear Discriminant Analysis), SVM (Support Vector Machine) and NN (the Neural Network based classifier). The obtained numerical results summarized in Table \ref{pagre2} show that proposed classification rule and neural networks method outperform the linear discriminant and support vector machines classifiers. As mentioned before,  the main advantages of the proposed method are its stability, convergence rate and fairly low computation time.\\

\begin{table}[h!]
		{\footnotesize{
	\begin{tabular}{llllll l l l l}
		\hline
		&     &       & \multicolumn{1}{l}{OM}& \multicolumn{1}{l}{LDA}  & \multicolumn{1}{l}{SVM} & \multicolumn{1}{l}{NN}\\ \hline
		$N$                          & $\sigma$   & $n_1$    & $CCR$   \vline& $CCR$  \vline& $CCR$   \vline& $CCR$  \\ \hline
		\multirow{1}{*}{5}
		& .05     &60  &95	&	65&		79&	\textbf{96}&	\\\
		&  $-$    &80  &95 	&	74&		82&	\textbf{97}&	\\\
		&  $-$    &100  &96 	&	73&		82&	\textbf{97}&	\\\
		\multirow{1}{*}{10}
		& .05     &60  &\textbf{99}	&	65&		79&	96&	\\\
		&  $-$    &80  &\textbf{97} 	&	74&		82&	97&	\\\
		&  $-$    &100  & \textbf{98}	&	73&		82&	97&	\\\		\multirow{1}{*}{5}
		& .15     &60   &93 &76&		84&	\textbf{95}&	\\\
		&  $-$    &80  &93&83&		90&	\textbf{97}&	\\\
		&  $-$    &100  &93 	&	72&		81&	\textbf{94}&	\\\
		\multirow{1}{*}{10}
		& .15     &60   &\textbf{95} &76&		84&	95&	\\\
		&  $-$    &80  &94&83&		90&	\textbf{97}&	\\\
		&  $-$    &100  &\textbf{95} 	&	72&		81&	94&	\\
			 \\ \hline
		\end{tabular}}}
			\centering
			\caption{Prediction results on a testing sample of size $80\%$ of the sample size of data generated using Example 3. The correct classification rate (CCR) in $\%$ for the proposed method (OM), Linear Discriminant Analysis (LDA), Support Vector Machine (SVM) and Neural Network (NN). The bold values highlight the best classification rates}
			\label{pagre2}
\end{table}


\subsection{Application to breast cancer cell lines regression and classification}




Most cancer patient die due to metastasis, and the early onset of this multi-step process is usually missed by current staging tumor modalities. Advanced techniques exist to enrich disseminated tumor cells from patient blood and bone marrow as cancer progression marker. However, these cells present high heterogeneity, only some of them can exhibit stem cell phenotype and tumor development potential, others can have plasticity potential to reprogram into cancer stem cells. So, detection and characterization are challenging due to lack of clear phenotypic markers. Therefore, there is a critical need to find new ways to anticipate and predict metastasis development at an early stage of patient care.
Cancer progression involves many cellular morphological effects, which have been revealed by biophysical studies. The relevance of the characterization of cancer cells by their bio-mechanical phenotype is attested by reports pointing out their physical alteration as reduced cell stiffness with invasiveness for lung, breast and colon cancers, while the deformability of circulating lymphocytes is reduced in the case of acute lymphoblastic leukemia. The physical properties even allow identifying different malignant breast epithelial cell lines by their deform-ability and their viscoelastic behavior.
Even though the analysis capability of cancer cells by their physical characteristics has been demonstrated, reports mostly compare different known states of cancer cells. So far to our knowledge, no prediction capability has been reported to detect cancer cells and evaluate their invasiveness only by bio-mechanical characterization. 
This application aims to use cell physical phenotyping to detect and categorize disseminating cells population by physical parameters (electrical measurements) using MEMS (Microelectromechanical systems) technology performing electrical single cell measurements.
 The MEMS devices \cite{ahmadian22,yamahataetl08,yamahataetl18} capture a cell for  stimulation and provides the information on the mechanical or electrical properties of the captured cell. They performed a compression protocol on each cell and measured changes in the resonance frequency and amplitude values for single-cell biophysical properties as a function of time in addition to the initial measurement on the cell dimension.\\
 The dataset analyzed here are derived from MEMS devices, they are composed of electrical properties (maximum values during the compression period, at 1 and 5 Mega hertz) of single cells from three different breast cancer cell lines in a controlled environment.
The compression assays on different breast cancer cell lines, U937, MCF7, SUM159-PT, give four electrical measurements (real and imaginary, at  1 and 5 Mega hertz, Figure \ref{fig2}). The size is obtained from 1 Mega hertz parameter (Figure \ref{fig2}). The three cell lines  has potential metastatic.
 The  SUM159-PT cell line has higher metastatic potential compared to the two other cell lines, SUM159-PT showed that cancer cells exhibit softer characteristics compared to their benign counterparts. The comparison of the average size (amplitude) between the cell lines (Figure  \ref{fig2}) indicated that SUM159-PT cell line (very aggressive and highly metastatic) was softer than MCF7  cell line (having lower metastatic potential). 
For the four electrical parameters and  size (amplitude) the  cell lines showed significant differences (Kruskal-Wallis comparison test have been done) between metastatic cell lines.\\

We apply the developed methodology to the cell dataset. To run our proposed estimator, we have first transformed (by usual dilation and translation techniques),  the set of  $4-$variate real data corresponding to the 1 and 5 Mega real and imaginary parameters values of the different cells into the square $[0,1]^2.$ Then, we have used these transformed data with total size almost equal to $3000,$ together with the standard Shepard scattered interpolation algorithm \eqref{Ineq2.3} with $p=3$ and   $n=n_1^2,$
for three couples $(N,n_1)= (5,30), (10,50), (15,70)$  and we have 
 computed fairly accurate numerical approximations of the values of cells at $n=n_1^2$ random i.i.d. sampling points in the $2-$dimensional unit square and following the two-dimensional Beta distribution with parameters $(\alpha+1,\alpha+1), \alpha=-0.5$.
First regression analysis has been done to explain the cell size (response variable) with the 5 Mega electrical parameters using the three cell types. We run the proposed model on $80 \%$ of the sample size for the construction of the estimators and the remaining $20 \%$ to validate the regression estimation by computing  the resulting mean squared and as well as the $R^2$. We compare our results with that of  the kernel, the support vector machine and the Random Forest regression estimations.
The results given in Table \ref{tableau3} show the same behavior as the results based on the simulated data. The proposed method is competitive compare to the above mentioned methods in particular the kernel method. However our method is less time consuming than the kernel regression estimate.
\begin{table}[h]
		{\footnotesize{
	\begin{tabular}{lll ll ll ll ll}
		\hline
		&            \multicolumn{2}{l}{OM}& \multicolumn{2}{l}{Kernel}  & \multicolumn{2}{l}{SVM} & \multicolumn{2}{l}{RF}\\ \hline
		$N$      & $MSE$   &$R^2$   \vline& $MSE$   &$R^2$  \vline& $MSE$   &$R^2$   \vline& $MSE$   &$R^2$  \\ \hline
		$5$      &5.97e-2 &9.80e-1 &	3.62e-2&	9.87e-1 	&	2.03e-1&	9.98e-1&		2.04e-1&	9.88e-1\\
		$10$     &4.95e-2 &9.83e-1		&	3.62e-2&	9.87e-1 	&	2.03e-1&	9.98e-1&		2.04e-1&	9.88e-1\\
	       $15$    &4.09e-2 &9.86e-1		&	3.62e-2&	9.87e-1 	&	2.03e-1&	9.98e-1&		2.04e-1&	9.88e-1\\
		\hline
	\end{tabular}}}
		\centering
		\caption{Validation of the cell size regression model estimation  on a test sample of size $20\%$ of the sample of size $n=2926$ cells using a training sample with the remaining $80\%$. The mean squared error (MSE) and $R^2$ on the test sample for the proposed method, the kernel, Support Vector Machine (SVM) and Random Forest (RF) methods.}
		\label{tableau3}
\end{table}\\
In the other hand, we run the proposed method (OM) and three other supervised learning methods to predict the population the cells belong to. These last methods are the  Linear Discriminant Analysis (LDA), Generalized Additive Models (GAM) and   Generalized Linear Models (GLM, logit).  For classification purpose, the cells of the classes MCF7 and U937 were assigned the integer values $1$ and $2,$ respectively. We have run  our proposed estimator
by using the same way as in the previous regression example with 
the couple $(N,n_1)=(5,30)$ and  transformed data with total size  almost
$1700.$
We have  computed our associated estimator $\widehat f^\alpha_N(\cdot),$ given by \eqref{Eq1.3}--\eqref{matrixG}.  The values are given in Table \ref{tableau4} where the overall classification accuracy (CR), precision (PR: the fraction of correct predictions for a certain class), recall (R: the fraction of instances of a class that were correctly predicted), the F1 (harmonic mean of precision and recall), are given on a test sample based on $20\%$ of the two different sample cells (MCF7 and U937) while the remaining  $80\%$ cells data are used for training. The results show that the proposed method has the second best overall correct classification rate compare to the best GAM model. When looking at the precision, the proposed method and GAM outperform when predicting the less metastatic cells (U937).\begin{table}[h]
	\centering
		{\footnotesize{
	\begin{tabular}{llll l l l l}
		\hline
		&      \multicolumn{1}{l}{OM}& \multicolumn{1}{l}{LDA}  & \multicolumn{1}{l}{GLM} & \multicolumn{1}{l}{GAM}\\ \hline
		$N$                       & $CR$   \vline& $CR$  \vline& $CR$   \vline& $CR$  \\ \hline
		\multirow{1}{*}{5}
		&   93  	&	89&		92.2&	94.8&	\\\
		\end{tabular}}}
		{\footnotesize{
	\begin{tabular}{llllll lll lll lll lll}
		\hline
		&     &       & \multicolumn{3}{l}{OM}& \multicolumn{3}{l}{LDA}  & \multicolumn{3}{l}{GLM} & \multicolumn{3}{l}{GAM}\\ \hline
		$N$                          &    &     & $PR$   &$R$  & $F1$  \vline& $PR$   &$R$  & $F1$ \vline& $PR$   &$R$  & $F1$ \vline& $PR$   &$R$  & $F1$\\ \hline
		\multirow{3}{*}{5}
		& MCF    &  &94 &94	&94	&	99 &86	&92	&	96&	92 & 94  &	96&	95&	96	\\\
		&  U937    &   &91 & 91	& 91	&	68 &97	&80	&	83&	93 & 87  &	90&	94&	92\ \\
		&      && &	&	&	&	& 	&	&	&	 &	&	&	 \\\hline
	\end{tabular}}}
		\centering
		\caption{Validation of the classification model estimation on a testing sample of size $20\%$ of the MCF7 and U937 cells. The overall classification accuracy (CR), precision (PR); the fraction of correct predictions for a certain class; the recall (R); fraction of instances of a class that were correctly predicted, the F1; harmonic mean of precision; for the proposed method, LDA, GAM and GLM}
		\label{tableau4}
\end{table}
  \\ 
The finite sample properties of this section shows that the proposed methodology is competitive to the well known parametric (GAM); nonparametric (kernel) and SVM  methods and less time consuming than the kernel regression and does not require any extra regularization or conditioning step.

\begin{figure}
\includegraphics[scale=0.45]{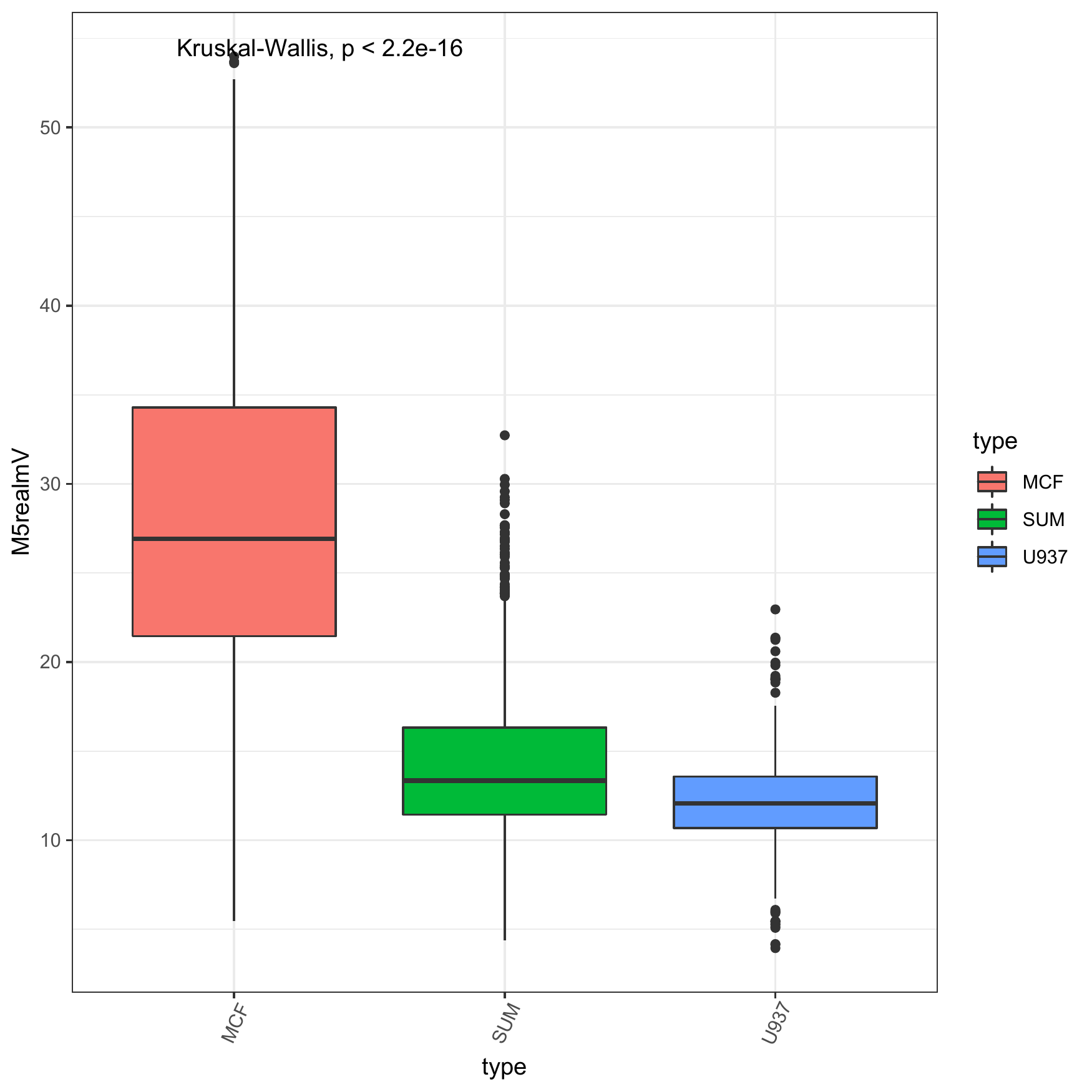}
\includegraphics[scale=0.45]{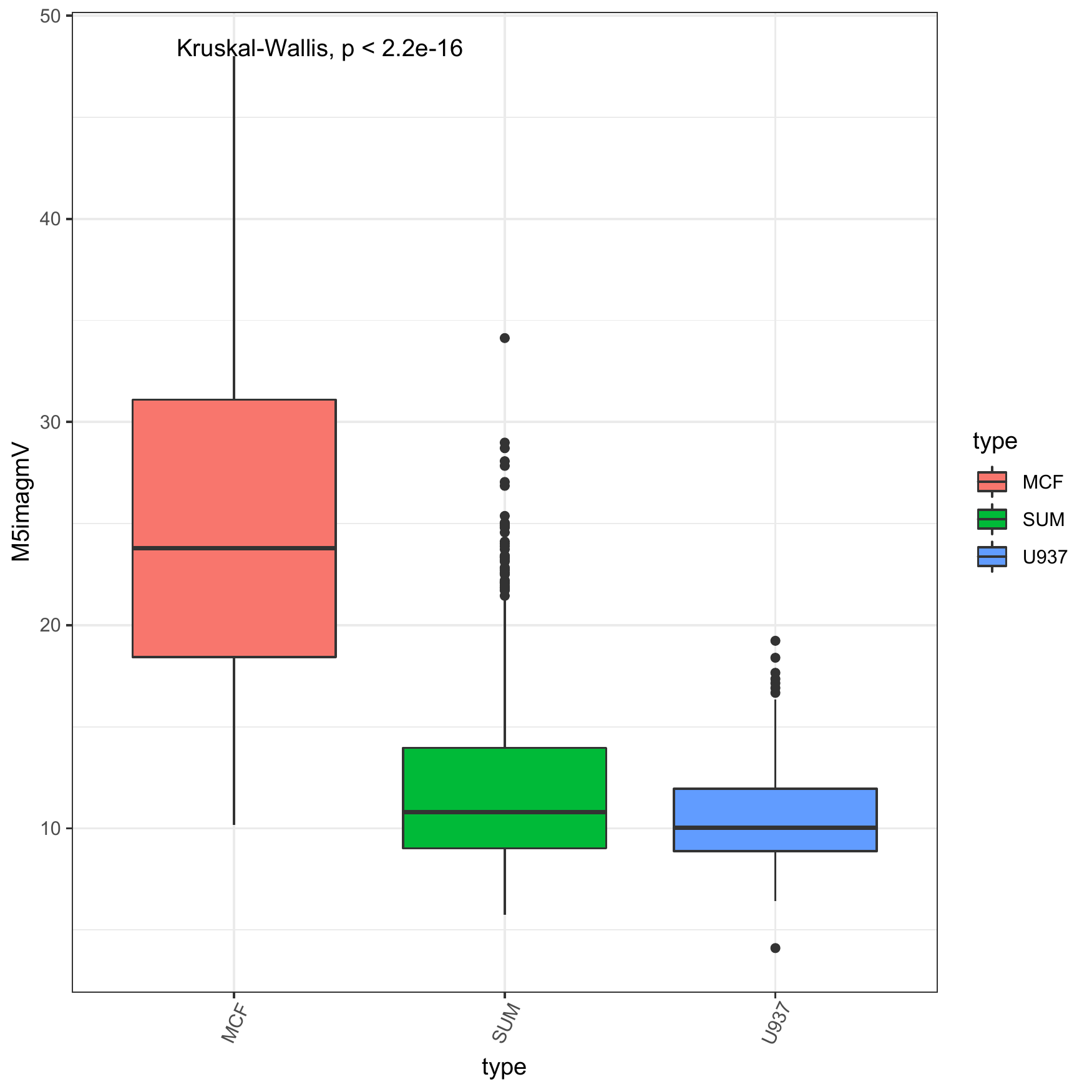}
\includegraphics[scale=0.45]{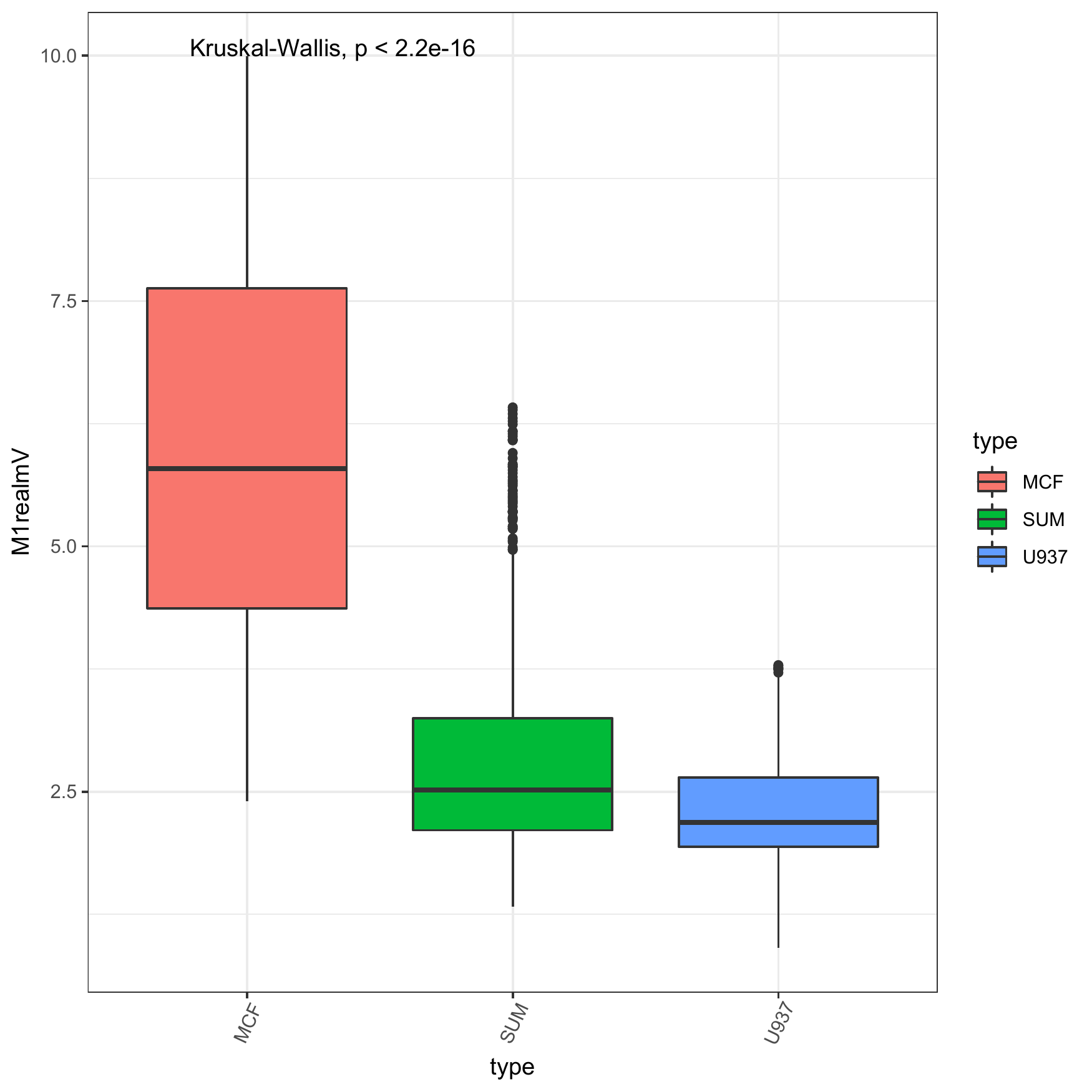}
\includegraphics[scale=0.45]{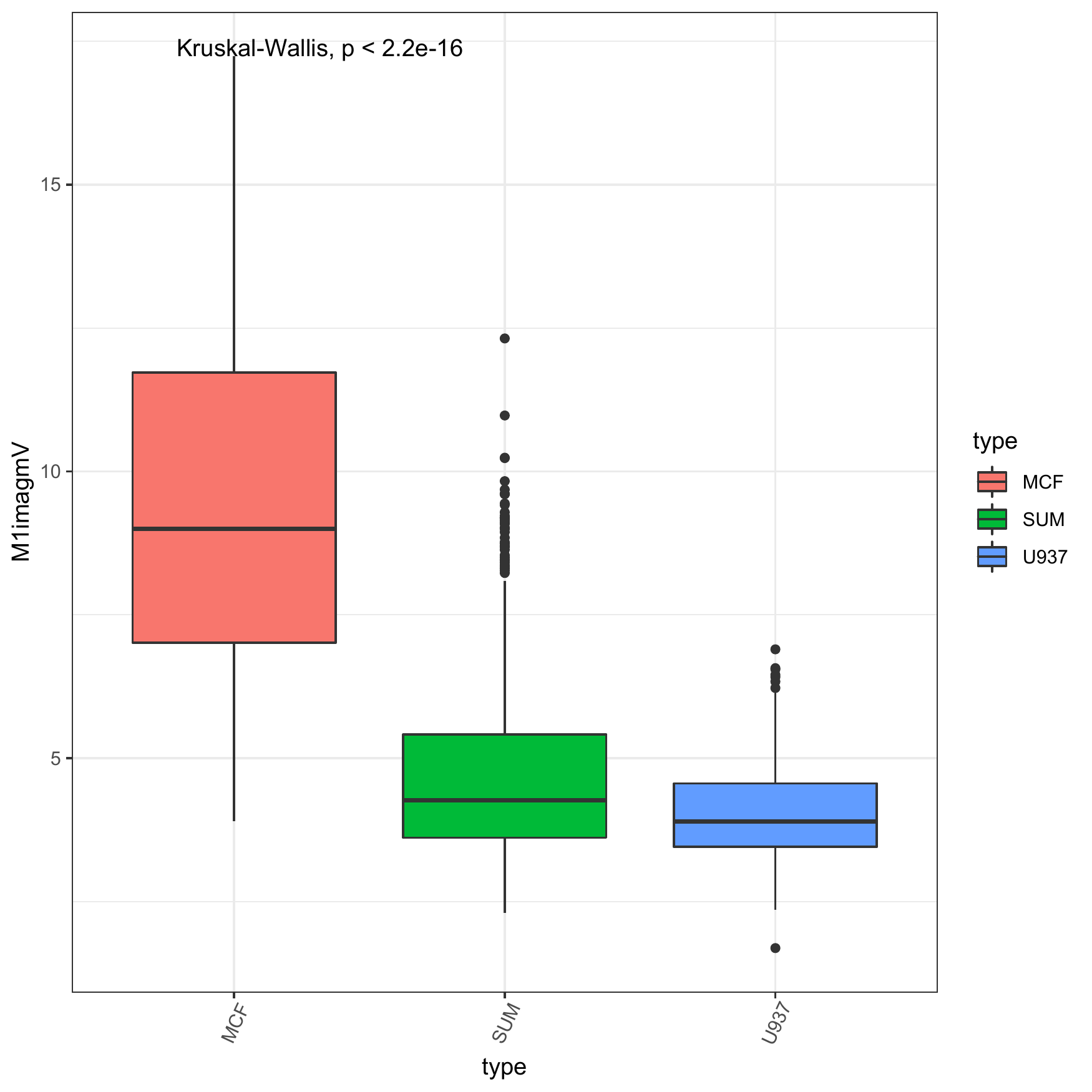}
\caption{Boxplot of the four electrical parameters; 5 Mega Hertz and 1 Mega Hertz, real and imaginary. Kruskal-Wallis test has been done to compare the parameters of the different cell types.}\label{fig1}
\end{figure}

\begin{center}
\begin{figure}
\includegraphics[scale=0.45]{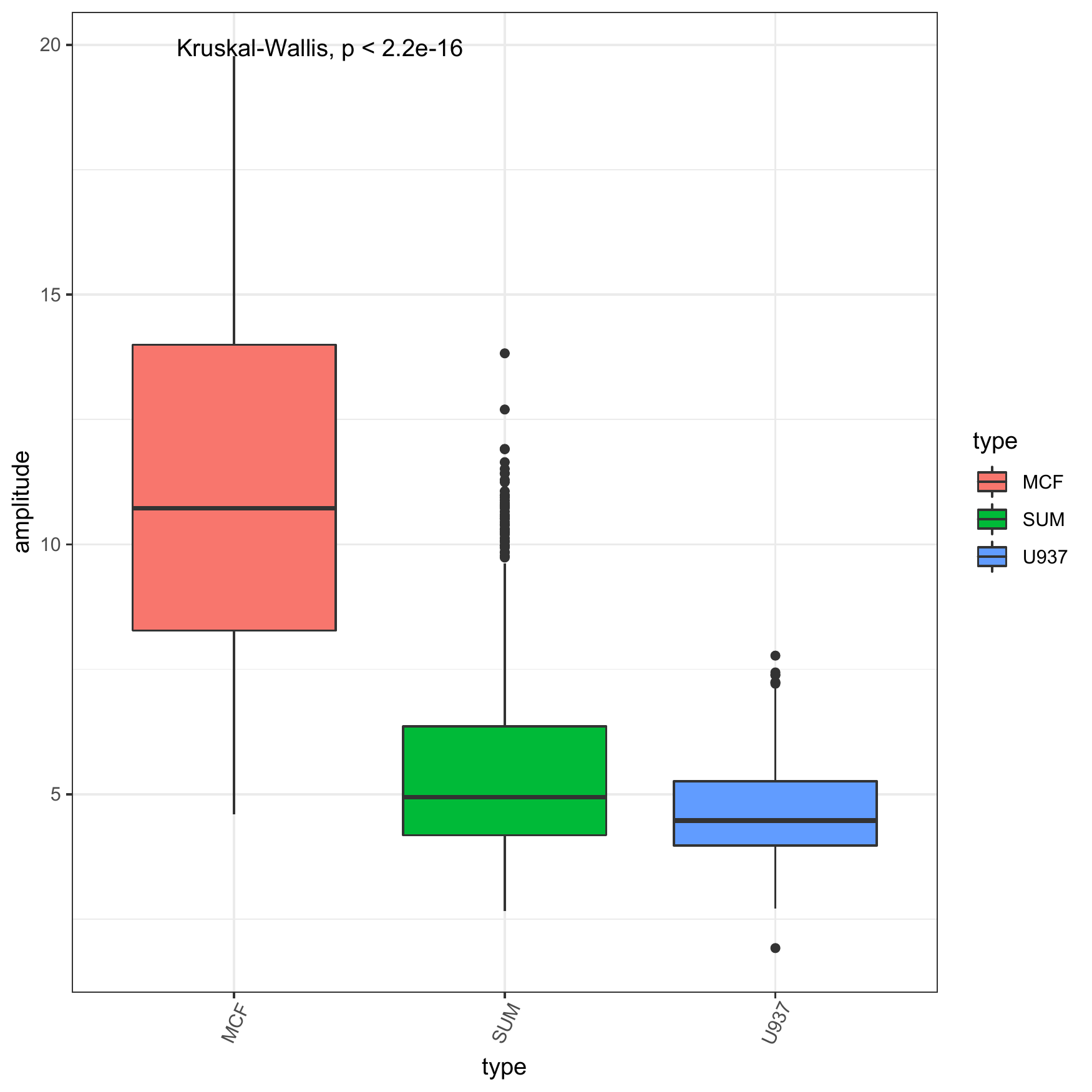}
\caption{Boxplot of the amplitude (size) of the different cell types with Kruskal-Wallis comparison test.}
\label{fig2}
\end{figure}
\end{center}

\section{Concluding remarks}

 We have proposed a least-squares multivariate nonparametric regression estimator by using the gPC (generalized Polynomial Chaos) principle. 
This estimator is given in terms of the tensor product of univariate Jacobi polynomials with parameters $\alpha=\beta \geq -\frac{1}{2}.$ In particular,
by using some spectral analysis results from the theory of positive definite random matrices, we have shown that this estimator is stable under the condition that the i.i.d. random training sampling points $\mathbf X_i$ follow a $d-$variate Beta distribution with parameters $(\alpha+1,\alpha+1)$ for each variable. Note that unlike many other least-squares nonparametric regression based estimators, the stability of our estimator does not require any extra regularization or conditioning step. \\
Also, we have performed an error analysis of our proposed estimator. More precisely, we have given an $L_2-$error as well as the $L_2-$risk error of this later. Moreover, we have studied its  convergence rate, when the regression function is assumed to belong to $d-$dimensional  isotropic Sobolev space $H^s(I^d),$ where $I=[0,1]$ and $s>0$ is the associated Sobolev smoothness exponent. In this case, we have shown that our estimator has the optimal min-max type convergence rate for $d-$dimensional regression problems and under the hypothesis that the regression function belongs to a functional space with some smoothness property.\\ 
Moreover, in the case of the frequently encountered case where the $n$ i.i.d. training sampling data points follow an unknown distribution, we have proposed to first apply a Shepard's type scattered interpolation technique to get fairly accurate approximations of the outputs at $n$ neighboring  random sampling sampling points following a $d-$variate Beta distribution. Numerical evidences indicate that even in the case of interpolated data, our proposed estimator still provides good results. We have performed numerical simulations on synthetic as well as real data. The results of these simulations indicate that the proposed estimator is competitive with some of  popular  multivariate regression  estimators from the literature, such as the smoothing kernel estimator, see for example \cite{Xiang}. Finally, we should mention that due to the its fairly heavy computational load, the estimator we have proposed in this work is rather adapted for small values of the dimension $d.$ Its  extension/adaptation to the case of  moderate or large values of $d,$ will be the subject of a future work.

\end{document}